\documentclass{article}
\usepackage[utf8]{inputenc}
\usepackage{amssymb,amsmath}
\usepackage[usenames, dvipsnames]{color}
\usepackage{amsthm}
\usepackage{centernot}
\usepackage{mathtools}
\usepackage{stmaryrd}
\newtheorem{theorem}{Theorem}
\newtheorem{remark}{Remark}
\newtheorem{proposition}{Proposition}

\newtheorem{corollary}{Corollary}
\newtheorem{question}{Problem}

\title{Representations of flat virtual braids which do not preserve the forbidden relations}
\author{V.~Bardakov \and B.~Chuzhinov \and I.~Emel'yanenkov \and  M.~Ivanov \and E.~Markhinina \and  T.~Nasybullov  \and S.~Panov  \and N.~Singh \and S.~Vasyutkin \and V.~Yakhin  \and A.~Vesnin} 
\date{~}

\begin{document}
\maketitle
\begin{abstract}
In the paper we construct a representation $\theta:FVB_n\to{\rm Aut}(F_{2n})$ of the flat virtual braid group $FVB_n$ on $n$ strands by automorphisms of the free group $F_{2n}$ with $2n$ generators which does not preserve the forbidden relations in the flat virtual braid group. This representation gives a positive answer to the problem formulated by V. Bardakov in the list of unsolved problems in virtual knot theory and combinatorial knot theory by R. Fenn, D. Ilyutko, L.~Kauffman and V. Manturov. Using this representation we construct a new group invariant for flat welded links. 

Also we find the set of normal generators of the groups $VP_n\cap H_n$ in $VB_n$, $FVP_n\cap FH_n$ in $FVB_n$, $GVP_n\cap GH_n$ in $GVB_n$, which play an important role in the study of the kernel of the representation $\theta$.

 ~
 
 ~
\end{abstract}

\section{Introduction}
The braid group $B_n$ on $n\geq2$ strands is the group with $n-1$ generators $\sigma_1,\sigma_2,\dots,\sigma_{n-1}$ and the following defining relations
\begin{align}
\label{def2}\sigma _i\sigma _j &= \sigma _j\sigma _i, && |i-j|\geq   2,\\
\label{def3}\sigma_i\sigma _{i+1}\sigma _i &= \sigma _{i+1}\sigma _i\sigma _{i+1},&& i = 1, 2, \dots, n-2. \end{align}
These groups were introduced by E.~Artin as a tool for working with classical knots and links. The famous Alexander theorem says that every classical link is equivalent to the closure of some braid, and the well-known Markov theorem describes braids which have equivalent closures.  

In 1999 L. Kauffman introduced the virtual knot theory which generalizes the classical knot theory \cite{Kau}. The virtual braid group $VB_n$ on $n\geq 2$ strands is the group which is obtained from the braid group $B_n$ by adding new generators $\rho_1,\rho_2,\dots,\rho_{n-1}$ and additional defining relations
\begin{align}
\label{def4}\rho_i^2&=1, &&i = 1, 2, \dots,n-1, \\
\label{def5}\rho _i\rho _j &= \rho _j\rho _i, && |i-j|\geq2,\\
\label{def6}\rho_i\rho _{i+1}\rho _i &= \rho _{i+1}\rho_i\rho _{i+1}, &&i = 1, 2, \dots. n-2,\\
\label{def7}\sigma _i\rho _j &= \rho _j\sigma _i, &&|i-j|\geq   2,\\
\label{def8}\rho_i\rho _{i+1}\sigma _i &= \sigma _{i+1}\rho_i\rho _{i+1}, && i = 1, 2, \dots, n-2.
\end{align}
These groups were introduced by L.~Kauffman as a tool for working with virtual knots and links. The analogues of the Alexander and the Markov theorems for virtual braids and links were formulated and proved by S.~Kamada in \cite{Kam}. Another form of the Markov theorem for virtual braids and links was introduced by Kauffman and Lambropoulou in \cite{KauLam}.

It is easy to verify that the elements $\rho_1,\rho_2,\dots,\rho_{n-1}$ generate the full symmetric group $S_n$ in $VB_n$. Also it is known that the elements $\sigma_1,\sigma_2,\dots,\sigma_{n-1}$ generate
the braid group $B_n$ in $VB_n$. In paper \cite{GuPoVi} it is proved that the relations
\begin{align}
\label{forbidden}\rho_i\sigma_{i+1}\sigma_i&=\sigma_{i+1}\sigma_i\rho_{i+1},&i=1,2,\dots,n-2,\\
\label{f2}\rho_{i+1}\sigma_i\sigma_{i+1}&=\sigma_i\sigma_{i+1}\rho_i,&i=1,2,\dots,n-2
\end{align}
do not hold in the group $VB_n$. These relations are called forbidden.The welded braid group $WB_n$ on $n$ strands is the quotient of the group $VB_n$ by the forbidden relation (\ref{forbidden}). If we add to the group $WB_n$ the second forbidden relations (\ref{f2}), then we get the unrestricted virtual braid group $UVB_n$ \cite{BarBelDom}. Welded and unrestricted virtual braid groups are used for working with welded links \cite{FRiRu} and fused links \cite{AuBeMeWa, Nas}, respectively.

The flat virtual braid group $FVB_n$ introduced in \cite{Kau2} is the quotient of the group $VB_n$ by the relations
\begin{align}
\label{def1}\sigma_i^2&=1, && i = 1, 2, \dots,n-1.
\end{align}
These groups are used for working with flat virtual knots and links, which were introduced in \cite{Kau}, and which model curves on surfaces. The notion of flat virtual links is the same to the notion of virtual strings introduced by V.~Turaev in \cite{Tur1}. In the papers \cite{Ves2, Ves1} invariants for flat virtual knots and links were used for constructing invariants for virtual knots and links. Note that if we look to the forbidden relations (\ref{forbidden}), (\ref{f2}) in the flat virtual braid group $FVB_n$, then we see that these relations are equivalent, hence, speaking about forbidden relations in  $FVB_n$ we will speak only about relations (\ref{forbidden}). The quotient of $FVB_n$ by relations (\ref{forbidden}) is called the flat welded braid group and is denoted by $FWB_n$. 

The Gauss virtual braid group $GVB_n$ is the quotient of the group $FVB_n$ by the relations
\begin{align*}
\sigma_i\rho_i=\rho_i\sigma_i, && i = 1, 2, \dots,n-1.
\end{align*}
These groups are used for studying Gauss virtual links, which are called ``homotopy classes of Gauss words'' by Turaev in \cite{Tur2}, and called ``free knots'' by Manturov in \cite{Man}. 

All together the groups $B_n, VB_n, WB_n, UVB_n, FVB_n, FWB_n, GVB_n$ are called the braid-like groups. Representations of braid-like groups play an important role in the study of these groups. Using the representations one can prove that various algorithmic problems (the word problem, the conjugacy problem etc) can be solved in these groups. Also using representations of braid-like groups one can construct invariants for the corresponding knot theories \cite{BarNas1, BarNas2}.

One of the first representations of braid-like groups is the Artin representation 
$$\varphi_A:B_n\to{\rm Aut}(F_n)$$ 
of the braid group $B_n$ by automorphisms of the free group $F_n$ with the free generators $x_1,x_2,\dots,x_n$. This representation maps the generator $\sigma_i$ of $B_n$ to the following automorphism of $F_n$
$$\varphi_A(\sigma_i):\begin{cases}
x_i\mapsto x_ix_{i+1}x_i^{-1},\\
x_{i+1}\mapsto x_i
\end{cases}$$
(hereinafter we assume that all generators which are not explicitly written in the automorphism, are fixed). The Artin representation is faithful. During the years a lot of representations of braid-like groups were investigated. For example there are several known representations of the virtual braid groups $VB_n$: the Silver-Williams representation $VB_n\to {\rm Aut}(F_{n}*\mathbb{Z}^{n+1})$ (see \cite{SilWil}), the Boden-Dies-Gaudreau-Gerlings-Harper-Nicas representation $VB_n\to {\rm Aut}(F_{n}*\mathbb{Z}^2)$ (see \cite{BDGGHN}),  the Kamada representation $VB_n\to {\rm Aut}(F_{n}*\mathbb{Z}^{2n})$ (see \cite{BN}), the representations $VB_n\to {\rm Aut}(F_n*\mathbb{Z}^{2n+1})$, $VB_n\to {\rm Aut}(F_n*\mathbb{Z}^n)$ of Bardakov-Mikhalchishina-Neshchadim (see \cite{BarMikNes, BarMikNes2}).

The following problem is formulated by V.~Bardakov in the list of unsolved problems in virtual knot theory and combinatorial knot theory \cite[\S8.3, Problem~4]{problems}.

~\\
\textbf{Problem.} Is there a representation $\varphi:FVB_n\to{\rm Aut}(F_m)$ for some $m$ which does not preserve the forbidden relations (\ref{forbidden})? 
~\\

In the present paper we answer this question positively. The main result of the paper is the following theorem. 

~\\
\textbf{Theorem~\ref{goodrepr}.} {\it Let  $\theta:\{\sigma_1,\sigma_2,\dots,\sigma_{n-1},\rho_1,\rho_2,\dots,\rho_{n-1}\}\to{\rm Aut}(F_{2n})$ be the map given by the following formulas
\begin{align*}
\theta(\sigma_i):
\begin{cases}
x_i \mapsto x_{i+1}y_{i+1},\\
x_{i+1} \mapsto x_iy_{i+1}^{-1},\\
\end{cases}&&
\theta(\rho_i):
\begin{cases}
x_i \mapsto x_{i+1},\\
x_{i+1} \mapsto x_i,\\
y_i \mapsto y_{i+1},\\
y_{i+1} \mapsto y_i.\\
\end{cases}
\end{align*}
Then $\theta$ induces the representation $\theta:FVB_n\to{\rm Aut}(F_{2n})$. Moreover, this representation does not preserve the forbidden relations (\ref{forbidden}).}
~\\

The paper is organized as follows. In Section~\ref{rabpur} we give the necessary preliminaries about the  virtual pure braid group $VP_n$, the Rabenda group $H_n$ and their images in the groups $FVB_n$, $GVB_n$. In Section~\ref{hitheta} we study the representation $\theta:FVB_n\to{\rm Aut}(F_{2n})$ introduced in Theorem~\ref{goodrepr}: in Section~\ref{newwwrepp} we prove that $\theta$ is indeed a representation, and in Section~\ref{kerrr} we study the kernel of this representation, in particular, we localize the kernel between two explicit subgroups of $FVB_n$. In Section~\ref{related} we study some representations related with the representation $\theta$: Section~\ref{linrepr} deals with linear representations of $FVB_n$, Section~\ref{gausrepr} studies the representation which is induced by $\theta$ to the group $GVB_n$, and Section~\ref{otherguys} is about some representations which are equivalent to $\theta$. In Section~\ref{intersectionss} we study the groups $VP_n\cap H_n$, $FVP_n\cap FH_n$, $GVP_n\cap GH_n$ which play an important role in the study of the kernel of $\theta$. Finally, in Section~\ref{invvvv} using the representation $\theta$ we introduce a group invariant for flat welded links.

\subsection*{Acknowledgements}
The work is supported by Mathematical Center in Akademgorodok under agreement No 075-2019-1675 with the Ministry of Science and Higher Education of the Russian Federation. A large number of authors is connected with the fact that the main results of the paper were obtained during the work on The First Workshop at the Mathematical Center in Akademgorodok (http://mca.nsu.ru/workshopen/). The results were obtained during the student program on knots and braid groups advised by V.~Bardakov, T.~Nasybullov and A.~Vesnin.

\section{Preliminaries}\label{rabpur}
In this section we give preliminaries about specific subgroups of the groups $VB_n$, $FVB_n$, $GVB_n$.

Denote by $\iota_1: VB_n\to S_n$ the homomorphism from the virtual braid group $VB_n$ to the symmetric group $S_n$, which maps the generators $\sigma_i,\rho_i$ of the virtual braid group $VB_n$ to the transposition $(i, i+1)$ for $i=1,2,\dots,n-1$. The kernel of this homomorphism is denoted by $VP_n$ and is called the virtual pure braid group on $n$ strands. The homomorphism $\iota_1: VB_n\to S_n$ induces homomorphisms $\iota_1:FVB_n\to S_n$, $\iota_1:GVB_n\to S_n$ which act in the same way mapping the generators $\sigma_i,\rho_i$ to the transposition $(i, i+1)$ for $i=1,2,\dots,n-1$. The kernel of the homomorphism $\iota_1:FVB_n\to S_n$ is denoted by $FVP_n$ and is called the flat virtual pure braid group on $n$ strands, the kernel of the homomorphism $\iota_1:GVB_n\to S_n$ is denoted by $GVP_n$ and is called the Gauss virtual pure braid group on $n$ strands.

For $i=1,2,\dots,n-1$ denote by $\lambda_{i,i+1}, \lambda_{i+1,i}$ the following elements from the virtual braid group
\begin{align}\notag\lambda_{i,i+1}=\rho_i\sigma_i^{-1},&&\lambda_{i+1,i}=\rho_i\lambda_{i,i+1}\rho_i=\sigma_i^{-1}\rho_i.
\end{align}
Also for $1\leq i<j-1\leq n-1$ denote by $\lambda_{i,j}, \lambda_{j,i}$ the following elements from the virtual braid group
\begin{align}\notag\lambda_{i,j}&=\rho_{j-1}\rho_{j-2}\dots\rho_{i+1}\lambda_{i,i+1}\rho_{i+1}\dots\rho_{j-2}\rho_{j-1},\\
\label{lambda}\lambda_{j,i}&=\rho_{j-1}\rho_{j-2}\dots\rho_{i+1}\lambda_{i+1,i}\rho_{i+1}\dots\rho_{j-2}\rho_{j-1}.
\end{align}
It is clear that the elements $\lambda_{i,j}$ and $\lambda_{j,i}$ for $i<j$ belong to the virtual pure braid group $VP_n$. The elements given by the same formulas can be considered in the groups $FVP_n$, $GVP_n$. The following proposition collects the results from \cite[Theorem~1]{Bar} and \cite[Propositions 5.1, 5.6]{BarBelDom}.
\begin{proposition}\label{vpngen}The groups $VP_n$, $FVP_n$, $GVP_n$ can be presented in the following way.
	\begin{enumerate}
		\item The group $VP_n$ admits a presentation with the generators $\lambda_{i,j}, 1 \le i \not= j \le n$, and the defining relations
		\begin{align*}
		\lambda_{i,j}\lambda_{k,l} &= \lambda_{k,l}\lambda_{i,j}, \\
		\lambda_{k,i}\lambda_{k,j}\lambda_{i,j} &= \lambda_{i,j}\lambda_{k,j}\lambda_{k,i},
		\end{align*}
		where distinct letters denote distinct indices.
		\item The group $FVP_n$ admits a presentation with the generators $\lambda_{i,j}, 1 \le i < j \le n$, and the defining relations
		\begin{align*}
		\lambda_{i,j}\lambda_{k,l} &= \lambda_{k,l}\lambda_{i,j}, \\
		\lambda_{k,i}\lambda_{k,j}\lambda_{i,j} &= \lambda_{i,j}\lambda_{k,j}\lambda_{k,i},
		\end{align*}
		where distinct letters denote distinct indices.
		\item The group $GVP_n$ admits a presentation with the generators $\lambda_{i,j}, 1 \le i < j \le n$, and the defining relations
		\begin{align*}
		\lambda_{i,j}^2&=1,\\
		\lambda_{i,j}\lambda_{k,l} &= \lambda_{k,l}\lambda_{i,j}, \\
		\lambda_{k,i}\lambda_{k,j}\lambda_{i,j} &= \lambda_{i,j}\lambda_{k,j}\lambda_{k,i},
		\end{align*}
		where distinct letters denote distinct indices.
	\end{enumerate}
\end{proposition}
The subgroup $S_n=\langle \rho_1,\rho_2,\dots,\rho_{n-1}\rangle$ of $VB_n$ acts on the group $VP_n$. The action of $S_n$ on the generators $\lambda_{i,j}$ of $VP_n$ is described in \cite[Lemma~1]{Bar} as follows.
\begin{proposition}\label{conjugation}
Let $\rho \in S_n=\langle \rho_1,\rho_2,\dots,\rho_{n-1}\rangle$ and $\lambda_{i,j}~(1 \le i \not= j \le n)$ be the generators of $VP_n$. Then $\rho \lambda_{i,j}\rho^{-1}=\lambda_{\rho(i),\rho(j)}$, where $\rho(i),\rho(j)$ denote the images of $i,j$ under the permutation $\rho$, respectively.
\end{proposition}
The group $FVP_n$ for $n=3$ has another nice presentation which is proved in \cite[Remark 5.2]{BarBelDom}.
\begin{proposition}\label{FVP31}
Let $a,b,c$ be the following elements from $FVP_3$
	\begin{equation}\label{fvp3}
\begin{cases}
a = \lambda_{2,3}\lambda_{1,3}=\rho_2\sigma_2\rho_2\rho_1\sigma_1\rho_2, \\
b = \lambda_{2,3}=\rho_2\sigma_2,\\
c = \lambda_{2,3}\lambda_{1,2}^{-1}=\rho_2\sigma_2\sigma_1\rho_1.\\
\end{cases}
\end{equation}
Then $FVP_3$ has the following presentation
$$FVP_3 = \langle a, b, c~|~[a,c] = 1 \rangle  = \langle a, c~|~[a,c] = 1 \rangle * \langle b \rangle=\mathbb{Z}^2*\mathbb{Z}.$$
\end{proposition}
Note that the elements $\lambda_{1,2}, \lambda_{1,3}, \lambda_{2,3}$ can be expressed via the elements $a,b,c$ from Proposition~\ref{FVP31} in the following way
\begin{align}\label{labc}
\lambda_{1,3}=b^{-1}a,&&
\lambda_{2,3}=b,&&
\lambda_{1,2}=c^{-1}b.
\end{align}

Denote by $\iota_2: VB_n\to S_n$ the homomorphism from the virtual braid group $VB_n$ to the symmetric group $S_n$, which maps the generator $\sigma_i$ to the unit element of $S_n$, and maps the generator $\rho_i$ to the transposition $(i, i+1)$ for $i=1,2,\dots,n-1$. The kernel of this homomorphism is denoted by $H_n$ and is called the Rabenda group on $n$ strands. The Rabenda group was introduced in \cite{Rab}, some of its properties were studied in \cite{BarBel}.  The homomorphism $\iota_2: VB_n\to S_n$ induces homomorphisms $\iota_2:FVB_n\to S_n$, $\iota_2:GVB_n\to S_n$ which act in the same way by mapping the generator $\sigma_i$ to the unit element of $S_n$, and the generator $\rho_i$ to the transposition $(i, i+1)$ for $i=1,2,\dots,n-1$. The kernel of the homomorphism $\iota_2:FVB_n\to S_n$ is denoted by $FH_n$, the kernel of the homomorphism $\iota_2:GVB_n\to S_n$ is denoted by $GH_n$.

For $i=1,2,\dots,n-1$ denote by $x_{i,i+1}, x_{i+1,i}$ the following elements from $VB_n$
\begin{align*}
x_{i,i+1}=\sigma_i,&&\
x_{i+1,i}=\rho_i\sigma_i\rho_i.
\end{align*}
Also for $1\leq i<j-1\leq n-1$ denote by $x_{i,j}, x_{j,i}$ the following elements from the virtual braid group
\begin{align*}
x_{i,j}&=\rho_{j-1}\ldots \rho_{i+1}\sigma_i\rho_{i+1}\ldots \rho_{j-1},\\
x_{j,i}&=\rho_{j-1}\ldots \rho_{i+1}\rho_i\sigma_i\rho_i\rho_{i+1}\ldots \rho_{j-1}.
\end{align*}
It is clear that the elements $x_{i,j}$ and $x_{j,i}$ for $i<j$ belong to the group $H_n$. Using the same formulas one can define the elements in the groups $FH_n$, $GH_n$. The following result says how the groups $H_n$, $FH_n$, $GH_n$ can be presented. The proof of this result can be found in \cite[Proposition~17]{BarBel}.
\begin{proposition}\label{rabendagen}The groups $H_n$, $FH_n$, $GH_n$ can be presented in the following way.
	\begin{enumerate}
		\item The group $H_n$ admits a presentation with the generators $x_{i,j}$, $1\le i\ne j\le n$, and the defining relations
		\begin{align*}
		x_{i,j}x_{k,l}&=x_{k,l}x_{i,j},\\
		x_{i,k}x_{k,j}x_{i,k}&=x_{k,j}x_{i,k}x_{k,j},
		\end{align*}
		where distinct letters denote disting indices.
		\item 	The group $FH_n$ admits a presentation with the generators $x_{i,j}$, $1\le i\neq j\le n$, and the defining relations
		\begin{align*}
		x_{i,j}^2&=1,\\
		x_{i,j}x_{k,l}&=x_{k,l}x_{i,j},\\
		x_{i,k}x_{k,j}x_{i,k}&=x_{k,j}x_{i,k}x_{k,j},
		\end{align*}
		where distinct letters denote disting indices.
		\item 	The group $GH_n$ admits a presentation with the generators $x_{i,j}$, $1\le i< j\le n$, and the defining relations
		\begin{align*}
		x_{i,j}^2&=1,\\
		x_{i,j}x_{k,l}&=x_{k,l}x_{i,j},\\
		x_{i,k}x_{k,j}x_{i,k}&=x_{k,j}x_{i,k}x_{k,j},
		\end{align*}
		where distinct letters denote disting indices.
	\end{enumerate}
\end{proposition}
The subgroup $S_n=\langle \rho_1,\rho_2,\dots,\rho_{n-1}\rangle$ of $VB_n$ acts on the group $H_n$. The action of $S_n$ on the generators $x_{i,j}$ of $H_n$ is described in \cite[Lemma~16]{BarBel} as follows.
\begin{proposition}\label{conjugation1}
	Let $\rho \in S_n=\langle \rho_1,\rho_2,\dots,\rho_{n-1}\rangle$ and $x_{i,j}~(1 \le i \not= j \le n)$ be the generators of $H_n$. Then $\rho x_{i,j}\rho^{-1}=x_{\rho(i),\rho(j)}$, where $\rho(i),\rho(j)$ denote the images of $i,j$ under the permutation $\rho$, respectively.
\end{proposition}
Propositions~\ref{vpngen},~\ref{conjugation},~\ref{rabendagen},~\ref{conjugation1} imply the decompositions of the groups $VB_n, FVB_n, GVB_n$ into the semidirect products 
\begin{align*}
VB_n&=VP_n\rtimes S_n,&&  FVB_n=FVP_n\rtimes S_n,&& GVB_n=GVP_n\rtimes S_n,\\
VB_n&=H_n\rtimes S_n, &&FVB_n=FH_n\rtimes S_n,&&GVB_n=GH_n\rtimes S_n.
\end{align*} 

For $1\leq i< j\leq n$ denote by $a_{i,j}$ the transposition $(i,j)$ in the symmetric group $S_n$. The following result is the direct consequence of \cite[Proposition~2.1]{BirKoLee}.

\begin{proposition}\label{permtrans}The group $S_n$ admits a presentation with generators $a_{t,s}$, $1\leq s<t\leq n$ and the defining relations
	\begin{align*}
	a_{t,s}^2&=1,&&1\leq s<t\leq n,\\
	a_{t,s}a_{r,q}&=a_{r,q}a_{t,s},&&(t-r)(t-q)(s-r)(s-q)>0,\\ 
	a_{t,s}a_{s,r}&=a_{t,r}a_{t,s}=a_{s,r}a_{t,r},&&n\geq t>s>r\geq 1.
	\end{align*}
\end{proposition}

\section{Representation of $FVB_n$ by automorphisms of $F_{2n}$}\label{hitheta}
In this section we introduce the representation $\theta:FVB_n\to{\rm Aut}(F_{2n})$ which does not preserve the forbidden relations and study some of its properties. All the actions in the paper are supposed to be right, i.~e. if $X$ is a set, and $f,g:X\to X$ are bijections of $X$, then $(fg)$ denotes the map $fg:X\to X$ which acts on an element $x\in X$ by the rule $(fg)(x)=g(f(x))$.
\subsection{Construction of $\theta$}\label{newwwrepp}
Let us consider the free group $F_{2n}$ with the free generators $x_1,x_2,\dots,x_n,y_1,y_2,\dots,y_n$. The following theorem provides a representation $\theta:FVB_n\to {\rm Aut}(F_{2n})$ which does not preserve the forbiden relations (\ref{forbidden}).
\begin{theorem}\label{goodrepr} Let  $\theta:\{\sigma_1,\sigma_2,\dots,\sigma_{n-1},\rho_1,\rho_2,\dots,\rho_{n-1}\}\to{\rm Aut}(F_{2n})$ be the map given by the following formulas
\begin{align}\label{mainrepr}
	\theta(\sigma_i):
	\begin{cases}
	x_i \mapsto x_{i+1}y_{i+1},\\
	x_{i+1} \mapsto x_iy_{i+1}^{-1},\\
	\end{cases}&&
	\theta(\rho_i):
	\begin{cases}
	x_i \mapsto x_{i+1},\\
	x_{i+1} \mapsto x_i,\\
	y_i \mapsto y_{i+1},\\
	y_{i+1} \mapsto y_i.\\
	\end{cases}
\end{align}
Then $\theta$ induces the representation $\theta:FVB_n\to{\rm Aut}(F_{2n})$. Moreover, this representation does not preserve the forbidden relations (\ref{forbidden}).
\end{theorem} 
\begin{proof}In order to show that the map $\theta$ given by (\ref{mainrepr}) induces the representation of the flat virtual braid group $FVB_n$, it is sufficient to show that $\theta$ preserves all defining relations (\ref{def2})-(\ref{def8}), (\ref{def1}) of the flat virtual braid group $FVB_n$.
	
For $i=1,2,\dots,n-1$ denote by $s_i,r_i\in {\rm Aut}(F_{2n})$ respectively the images of the elements $\sigma_i,\rho_i\in FVB_n$ under the map $\theta$. The fact that $\theta$ preserves relations (\ref{def4}), (\ref{def1}) means that the equalities $r_i^2=id$, $s_i^2=id$ hold for all $i=1,2,\dots,n-1$. These equalities are obvious. The fact that $\theta$ preserves relations (\ref{def2}), (\ref{def5}), (\ref{def7}) means that $s_is_j=s_js_i$, $r_ir_j=r_jr_i$ and $s_ir_j=r_js_i$ for all indices $i,j$ such that $|i-j|\geq2$. These equalities are easy to check. Indeed, automorphisms $s_i$, $r_i$ act identically on all generators of $F_{2n}$ except for $x_i, x_{i+1}, y_i, y_{i+1}$. Then for $|i-j|\ge2$  the automorphism $s_i$ acts identically on  $x_j, x_{j+1}, y_j, y_{j+1}$, and $s_j$ acts identically on  $x_i, x_{i+1}, y_i, y_{i+1}$. Therefore $s_is_j=s_js_i$. Similarly the equalities $r_ir_j=r_jr_i$ and $s_ir_j=r_js_i$ hold for $|i-j|\ge2$.

The fact that $\theta$ preserves relations (\ref{def3}), (\ref{def6}), (\ref{def8}) means that for all $i=1,2,\dots,n-2$ the equalities $s_is_{i+1}s_i=s_{i+1}s_is_{i+1}$, $r_ir_{i+1}r_i=r_{i+1}r_ir_{i+1}$, $s_ir_{i+1}r_i=r_{i+1}r_is_{i+1}$ hold. Let us prove the equality  $s_is_{i+1}s_i=s_{i+1}s_is_{i+1}$ in details. By the definition of the elements $s_1,s_2,\dots,s_{n-1}$ the automorphisms  $s_is_{i+1}s_i$, $s_{i+1}s_is_{i+1}$ have the following forms.
\begin{align*}
s_is_{i+1}s_i &:
\begin{cases}
x_i\overset{s_i}{\longmapsto} x_{i+1}y_{i+1}\overset{s_{i+1}}{\longmapsto} x_{i+2}y_{i+2}y_{i+1}\overset{s_{i}}{\longmapsto}x_{i+2}y_{i+2}y_{i+1},\\
x_{i+1}\overset{s_{i}}{\longmapsto} x_{i}y_{i+1}^{-1}\overset{s_{i+1}}{\longmapsto} x_{i}y_{i+1}^{-1}\overset{s_{i}}{\longmapsto} x_{i+1},\\
x_{i+2}\overset{s_{i}}{\longmapsto} x_{i+2}\overset{s_{i+1}}{\longmapsto} x_{i+1}y_{i+2}^{-1}\overset{s_{i}}{\longmapsto} x_iy_{i+1}^{-1}y_{i+2}^{-1},\\
\end{cases} \\
s_{i+1}s_is_{i+1} &:
\begin{cases}
x_i\overset{s_{i+1}}{\longmapsto}x_i\overset{s_{i}}{\longmapsto} x_{i+1}y_{i+1}\overset{s_{i+1}}{\longmapsto} x_{i+2}y_{i+2}y_{i+1},\\
x_{i+1}\overset{s_{i+1}}{\longmapsto} x_{i+2}y_{i+2}\overset{s_{i}}{\longmapsto} x_{i+2}y_{i+2}\overset{s_{i+1}}{\longmapsto} x_{i+1},\\
x_{i+2}\overset{s_{i+1}}{\longmapsto} x_{i+1}y_{i+2}^{-1}\overset{s_{i}}{\longmapsto} x_iy_{i+1}^{-1}y_{i+2}^{-1}\overset{s_{i+1}}{\longmapsto} x_iy_{i+1}^{-1}y_{i+2}^{-1},\\
\end{cases} 
\end{align*}
therefore, $s_is_{i+1}s_i=s_{i+1}s_is_{i+1}$ for all $i=1,2,\dots,n-2$. The equalities $r_ir_{i+1}r_i=r_{i+1}r_ir_{i+1}$, $s_ir_{i+1}r_i=r_{i+1}r_is_{i+1}$ follow from the following formulas which by the analogy to the equality $s_is_{i+1}s_i=s_{i+1}s_is_{i+1}$ follow from the direct calculations.
\begin{align*}
r_ir_{i+1}r_i &:
\begin{cases}
x_i\mapsto x_{i+2},\\
x_{i+1}\mapsto x_{i+1},\\
x_{i+2}\mapsto x_i,\\
y_{i} \mapsto y_{i+2},\\
y_{i+1} \mapsto y_{i+1},\\
y_{i+2} \mapsto y_{i},\\
\end{cases} &\qquad
r_{i+1}r_ir_{i+1} &:
\begin{cases}
x_i\mapsto x_{i+2},\\
x_{i+1}\mapsto x_{i+1},\\
x_{i+2}\mapsto x_i,\\
y_{i} \mapsto y_{i+2},\\
y_{i+1} \mapsto y_{i+1},\\
y_{i+2} \mapsto y_{i},\\
\end{cases} \\ 
r_{i+1}r_is_{i+1} &:
\begin{cases}
x_i\mapsto x_{i+2}y_{i+2},\\
x_{i+1}\mapsto x_{i+1}y_{i+2}^{-1},\\
x_{i+2}\mapsto x_i,\\
y_{i} \mapsto y_{i+1},\\
y_{i+1} \mapsto y_{i+2},\\
y_{i+2} \mapsto y_{i},\\
\end{cases} &\qquad
s_ir_{i+1}r_i &:
\begin{cases}
x_i\mapsto x_{i+2}y_{i+2},\\
x_{i+1}\mapsto x_{i+1}y_{i+2}^{-1},\\
x_{i+2}\mapsto x_i,\\
y_{i}\mapsto y_{i+1},\\
y_{i+1}\mapsto y_{i+2},\\
y_{i+2}\mapsto y_{i}.\\
\end{cases} 
\end{align*}
Therefore  $\theta$ preserves all defining relations (\ref{def2})-(\ref{def8}), (\ref{def1}) of $FVB_n$, and hence $\theta$ induces the representation $\theta:FVB_n\to{\rm Aut}(F_{2n})$. Finally, $r_is_{i+1}s_i(y_i)=y_{i+1}\neq y_i=s_{i+1}s_ir_{i+1}(y_i)$, therefore $\theta$ does not preserve the forbidden relations.    
\end{proof}

Note that the representation $\theta:FVB_n\to {\rm Aut}(F_{2n})$ introduced in Theorem~\ref{goodrepr} gives a positive answer to the problem \cite[\S8.3, Problem~4]{problems}.

\subsection{Properties of ${\rm ker}(\theta)$}\label{kerrr}
The following proposition says, that if $n=2$, then the representation $\theta$ is faithful. 
\begin{proposition}
The representation $\theta:FVB_2\to {\rm Aut}(F_4)$ is faithful.
\end{proposition}
\begin{proof}
The group $FVB_2=\langle \sigma_1,\rho_1~|~\sigma_1^2=\rho_1^2=1\rangle$ is the infinite dihedral group. Therefore, every element $A$ of this group can be uniquely written in the form $A=\rho_1^{\varepsilon}(\sigma_1\rho_1)^n$ for some $\varepsilon\in\{0,1\}$, $n\in \mathbb{Z}$. Denote by $s_1=\theta(\sigma_1)$, $r_1=\theta(\rho_1)$, and let $A=\rho_1^{\varepsilon}(\sigma_1\rho_1)^n$ be an element from ${\rm ker}(\theta)$. Since $A\in {\rm ker}(\theta)$ we have the equality $y_1=\theta(A)(y_1)=r_1^{\varepsilon}(s_1r_1)^{n}(y_1)$. Since $r_1$ permutes $y_1$ and $y_2$, and $s_1$ fixes $y_1,y_2$, from the last equality follows that $r_1^{\varepsilon+n}(y_1)=y_1$. It means that either $\varepsilon=0$ and $n$ is even, or $\varepsilon=1$ and $n$ is odd. Let us consider these two cases separately.

Let $\varepsilon=0$ and $n=2k$ be even. The automorphism $(s_1r_1)^2$ has the following form
$$
(r_1s_1)^2:\begin{cases}
x_1\overset{r_1}{\longmapsto}x_2\overset{s_1}{\longmapsto}x_1y_2^{-1}\overset{r_1}{\longmapsto}x_2y_1^{-1}\overset{s_1}{\longmapsto}x_1y_2^{-1}y_1^{-1},\\
x_2\overset{r_1}{\longmapsto}x_1\overset{s_1}{\longmapsto}x_2y_2\overset{r_1}{\longmapsto}x_1y_1\overset{s_1}{\longmapsto}x_2y_2y_1.
\end{cases}
$$
Therefore, the automorphism $\theta(A)=(r_1s_1)^{2k}=((r_1s_1)^2)^{k}$ has the form
\begin{equation}\label{simpleoddpower}
(r_1s_1)^{2k}:\begin{cases}
x_1\mapsto x_1(y_1y_2)^{-k},\\
x_2\mapsto x_2(y_2y_1)^k.
\end{cases}
\end{equation}
It is clear that this automorphism is identity if and only if $k=0$, i.~e. in the case when $A=1$. 

Let $\varepsilon=1$ and $n=2k+1$ be odd. In this case $\theta(A)=r_1s_1r_1 (r_1s_1)^{2k}$. 
From equality (\ref{simpleoddpower}) follows that this automorphism maps the element $x_1$ to the element $x_2(y_2y_1)^ky_1^{-1}$, which is never equal to $x_1$
$$
r_1s_1r_1 (r_1s_1)^{2k}:
x_1\overset{r_1}{\longmapsto}x_2\overset{s_1}{\longmapsto}x_1y_2^{-1}\overset{r_1}{\longmapsto}x_2y_1^{-1}\xmapsto{(r_1s_1)^{2k}}x_2(y_2y_1)^ky_1^{-1}.
$$
Therefore we proved, that the element $A=\rho_1^{\varepsilon}(\sigma_1\rho_1)^n$ belongs to the kernel of $\theta$ if and only if $\varepsilon=0$ and $k=0$, i.~e. in the case when $A=1$. It means that the representation $\theta:FVB_2\to {\rm Aut}(F_4)$ is faithful.
\end{proof}

The following proposition says that the representation  $\theta:FVB_n\to {\rm Aut}(F_{2n})$ introduced in Theorem~\ref{goodrepr} is not faithful, i.~e. the kernel ${\rm ker}(\theta)$ is not trivial.
\begin{proposition}\label{nontrivker}If $n\geq 3$, then ${\rm ker}(\theta)\neq1$.
\end{proposition}
\begin{proof}Consider the element $A=(\rho_1\sigma_2)^6$ in $FVB_n$, and let us prove that this element is the non-trivial element from $FVB_n$ which belongs to the kernel of $\theta$. Denote by $r_1=\theta(\rho_1)$, $s_2=\theta(\sigma_2)$. Then from the direct calculations it follows that the automorphisms $(r_1s_2)^3$ has the following form
$$(r_1s_2)^3:\begin{cases}
x_1\overset{r_1}{\longmapsto}x_2\overset{s_2}{\longmapsto}x_3y_3\overset{r_1}{\longmapsto}x_3y_3\overset{s_2}{\longmapsto}x_2\overset{r_1}{\longmapsto}x_1\overset{s_2}{\longmapsto}x_1,\\
x_2\overset{r_1}{\longmapsto}x_1\overset{s_2}{\longmapsto}x_1\overset{r_1}{\longmapsto}x_2\overset{s_2}{\longmapsto}x_3y_3\overset{r_1}{\longmapsto}x_3y_3\overset{s_2}{\longmapsto}x_2,\\
x_3\overset{r_1}{\longmapsto}x_3\overset{s_2}{\longmapsto}x_2y_3^{-1}\overset{r_1}{\longmapsto}x_1y_3^{-1}\overset{s_2}{\longmapsto}x_1y_3^{-1}\overset{r_1}{\longmapsto}x_2y_3^{-1}\overset{s_2}{\longmapsto}x_3,\\
y_1 \overset{r_1}{\longmapsto} y_2 \overset{s_2}{\longmapsto}y_2 \overset{r_1}{\longmapsto} y_1 \overset{s_2}{\longmapsto}y_1 \overset{r_1}{\longmapsto}y_2 \overset{s_2}{\longmapsto}y_2,\\
y_2 \overset{r_1}{\longmapsto} y_1 \overset{s_2}{\longmapsto}y_1 \overset{r_1}{\longmapsto} y_2 \overset{s_2}{\longmapsto}y_2 \overset{r_1}{\longmapsto}y_1 \overset{s_2}{\longmapsto}y_1,\\
y_3 \overset{r_1}{\longmapsto} y_3 \overset{s_2}{\longmapsto}y_3 \overset{r_1}{\longmapsto} y_3 \overset{s_2}{\longmapsto}y_3 \overset{r_1}{\longmapsto}y_3 \overset{s_2}{\longmapsto}y_3.
\end{cases}
$$
From this formula it is clear that $\theta(A)=((r_1s_2)^3)^2=id$, i.~e. $A$ belongs to ${\rm ker}(\theta)$.

Let us prove that the element $A=(\rho_1\sigma_2)^6$ is not trivial. Since there is an ascending series of groups
$$FVB_n\geq FVB_{n-1}\geq\dots\geq FVB_3\geq FVP_3,$$
in order to prove that the element $A$ is not trivial, it is enough to prove that $A$ is not trivial in $FVP_3$. For doing this let us rewrite the element $A$ in terms of the generators $a,b,c$ given in Proposition~\ref{FVP31}. Using Proposition~\ref{conjugation} and formulas (\ref{lambda}), (\ref{labc}) the element $A$ can be rewritten in the following form
\begin{align*}
A&=(\rho_1\sigma_2\rho_1\sigma_2\rho_1\sigma_2)^2\\
&=(\rho_1\rho_2(\rho_2\sigma_2)\rho_1\rho_2(\rho_2\sigma_2)\rho_1\rho_2(\rho_2\sigma_2))^2\\
&=(\rho_1\rho_2\lambda_{2,3}\rho_1\rho_2\lambda_{2,3}\rho_1\rho_2\lambda_{2,3})^2\\
&=(\rho_1\rho_2\rho_1\rho_2\lambda_{1,2}\lambda_{2,3}\rho_1\rho_2\lambda_{2,3})^2\\
&=(\rho_1\rho_2\rho_1\rho_2\rho_1\rho_2\lambda_{1,3}^{-1}\lambda_{1,2}\lambda_{2,3})^2\\
&=(\lambda_{1,3}^{-1}\lambda_{1,2}\lambda_{2,3})^2\\
&=(a^{-1}bc^{-1}b^2)^2.
\end{align*}
It is clear that this element is not trivial in the group $FVP_3 = \langle a, b, c ~|~ [a,c] = 1 \rangle$.
\end{proof}
\begin{remark}\label{bigkernel}{\rm Proposition~\ref{nontrivker} says that if $n\geq3$, then the element $(\rho_1\sigma_2)^6$ is a non-trivial element in ${\rm ker}(\theta)$. By the analogy to Proposition~\ref{nontrivker} it is easy to show that the element $(\rho_i\sigma_{i+1})^6$ is also a non-trivial element in ${\rm ker}(\theta)$ for each $i=1,2,\dots,n-2$. This observation implies that the normal closure of the subgroup $\langle(\rho_i\sigma_{i+1})^6~|~i=1,2,\dots,n-2\rangle$ in $FVB_n$ belongs to the kernel ${\rm ker}(\theta)$.}
\end{remark}
Remark~\ref{bigkernel} says that the kernel of the representation $\theta$ is quite big. Despite this fact, the kernel ${\rm ker}(\theta)$ has trivial intersections with some natural subgroups of $VB_n$. For example, by Proposition~\ref{FVP31} the group $FVP_3$ has the form
$$FVP_3=\langle a,b,c~|~[a,c]=1\rangle,$$
where $a =\rho_2\sigma_2\rho_2\rho_1\sigma_1\rho_2$, $b = \rho_2\sigma_2$,  $c =\rho_2\sigma_2\sigma_1\rho_1$. This group has an abelian subgroup $H$ generated by the elements $a,c$. Let us show that ${\rm ker}(\theta)\cap H=1$. Denote by $\theta(a)=\alpha$, $\theta(c)=\beta$. From direct calculations and the definition of the representation $\theta:FVB_n\to{\rm Aut}(F_{2n})$ it follows that the automorphisms $\alpha,\beta$ act on the generators of $F_{2n}$ by the following rules
\begin{align*}
 \alpha : 
\begin{cases}
x_1 \rightarrow x_1 y_3^{-1},\\
x_2 \rightarrow x_2 y_1^{-1},\\
x_3 \rightarrow x_3 y_3 y_1,\\
y_1 \rightarrow y_3,\\
y_2 \rightarrow y_1,\\
y_3 \rightarrow y_2,
\end{cases}&&
 \beta :
\begin{cases}
x_1 \rightarrow x_1 y_1,\\
x_2 \rightarrow x_2 y_1^{-1} y_3^{-1},\\
x_3 \rightarrow x_3 y_3,\\
y_1 \rightarrow y_2,\\
y_2 \rightarrow y_3,\\
y_3 \rightarrow y_1.
\end{cases}
\end{align*}
From these equalities and direct calculations it follows that for an arbitrary integer $r$ the automorphism $(\alpha\beta)^r$ acts on the generators of $F_{2n}$ by the following rule
$$
(\alpha \beta)^r :
\begin{cases}
x_1 \rightarrow x_1,\\
x_2 \rightarrow x_2 (y_2 y_3 y_1)^{-r},\\
x_3 \rightarrow x_3 (y_3 y_1 y_2)^r,\\
y_1 \rightarrow y_1,\\
y_2 \rightarrow y_2,\\
y_3 \rightarrow y_3.
\end{cases}
$$
Let $A=a^rc^s=(ac)^rc^{s-r}$ be an element from ${\rm ker}(\theta)\cap H$. Since $A$ belongs to ${\rm ker}(\theta)$, the automorphism $\theta(A)=(\alpha\beta)^r\beta^{s-r}$ has to fix the generator $x_1$. This fact implies that  $s-r = 0$ since otherwise we have $(\alpha \beta)^r \beta^{s-r} (x_1)= \beta^{s-r} (x_1)\neq x_1$. If $s-r = 0$, and $r \neq 0$, then $(\alpha \beta)^r (x_2)=x_2(y_2y_3y_1)^{-r} \neq x_2$, what contradicts the fact that $A$ belongs to ${\rm ker}(\theta)$. Therefore $s=r =0$ and ${\rm ker}(\theta)\cap H = 1$.  

In the general situation we have the following restriction on the kernel ${\rm ker}(\theta)$ of the representation $\theta:FVB_n\to{\rm Aut}(F_{2n})$.
\begin{proposition}\label{kerisinintern}If $n\geq 3$, then ${\rm ker}(\theta)\leq FVP_n\cap FH_n$.
\end{proposition}
\begin{proof}Let $A$ be an arbitrary element from ${\rm ker}(\theta)$, and let us prove that $A\in FVP_n\cap FH_n$. Recall that the representation $\theta$ acts on the generators $\sigma_1,\sigma_2,\dots,\sigma_{n-1},\rho_1,\rho_2,\dots,\rho_{n-1}$ of $FVB_n$ by the following rules
	\begin{align}\label{mainrepr1}
	\theta(\sigma_i):
	\begin{cases}
	x_i \mapsto x_{i+1}y_{i+1},\\
	x_{i+1} \mapsto x_iy_{i+1}^{-1},\\
	\end{cases}&&
	\theta(\rho_i):
	\begin{cases}
	x_i \mapsto x_{i+1},\\
	x_{i+1} \mapsto x_i,\\
	y_i \mapsto y_{i+1},\\
	y_{i+1} \mapsto y_i.\\
	\end{cases}
	\end{align}
	
Let us prove that $A$ belongs to $FH_n$. Denote by $Y=\{y_1,y_2,\dots,y_n\}$	the subset in $F_{2n}$.  From formulas (\ref{mainrepr1}) it is clear that the set $Y$ is $\theta(FVB_n)$-invariant. Denote by $\varphi:VB_n\to {\rm Sym}(Y)=S_n$ the representation induced by $\theta$. From formulas (\ref{mainrepr1}) it follows that for $i=1,2,\dots,n-1$ the map $\varphi$ sends the generator $\sigma_i$ to the trivial permutation of $Y=\{y_1,y_2,\dots,y_n\}$, and sends the generator $\rho_i$ to the transposition $(y_i,~y_{i+1})$ of $Y=\{y_1,y_2,\dots,y_n\}$. Comparing $\varphi$ with $\iota_2$ we see that $\varphi=\iota_2$. Since $A$ is an element from  ${\rm ker}(\theta)$, and the representation  $\varphi:VB_n\to {\rm Sym}(Y)=S_n$ is induced by $\theta$, it is clear that $A$ is an element from ${\rm ker}(\varphi)={\rm ker}(\iota_2)=FH_n$.

Let us prove that $A$ belongs to $FVP_n$. Denote by $G$ the subgroup of $F_{2n}$ which is the normal closure of the subgroup $\langle y_1,y_2,\dots,y_n\rangle$ in $F_{2n}$. From formulas (\ref{mainrepr1}) it is clear that the subgroup $G$ is $\theta(FVB_n)$-invariant. Denote by $\psi:VB_n\to {\rm Aut}(F_{2n}/G)={\rm Aut}(F_{n})$ the representation induced by $\theta$, where  $F_{2n}/G=F_n=\langle x_1,x_2,\dots,x_n\rangle$. From formulas (\ref{mainrepr1}) it follows that for $i=1,2,\dots,n-1$ the map $\psi$ acts on the generators of $FVB_n$ by the following formulas
\begin{align*}
\psi(\sigma_i):
\begin{cases}
x_i \mapsto x_{i+1},\\
x_{i+1} \mapsto x_i,\\
\end{cases}&&
\psi(\rho_i):
\begin{cases}
x_i \mapsto x_{i+1},\\
x_{i+1} \mapsto x_i.
\end{cases}
\end{align*}
These maps act on the set $\{x_1,x_2,\dots,x_n\}$, sending the generators  $\sigma_i, \rho_i$ to the transposition $(x_i,x_{i+1})$ of $X=\{x_1,x_2,\dots,x_n\}$. Comparing $\psi$ with $\iota_1$ we see that $\psi=\iota_1$. Since $A$ is an element from  ${\rm ker}(\theta)$, and the representation  $\psi:VB_n\to {\rm Aut}(F_n)$ is induced by $\theta$, it is clear that $A$ is an element from ${\rm ker}(\psi)={\rm ker}(\iota_1)=FVP_n$.
\end{proof}
Propositions~\ref{nontrivker}, \ref{kerisinintern} imply the following corollary, which localize the kernel of the representation~$\theta$.
\begin{corollary}\label{corker}For $n\geq 3$ denote by $N$ the normal closure of the subgroup $$\langle(\rho_i\sigma_{i+1})^6~|~i=1,2,\dots,n-2\rangle$$ 
	in $FVB_n$. Then $N\leq {\rm ker}(\theta)\leq FVP_n\cap FH_n$.
\end{corollary}
\begin{question}{\rm Find the kernel ${\rm ker}(\theta)$.} 
\end{question}
\section{Representations related with $\theta$}\label{related}
In this section we study several representations which are related with the representation $\theta$ given in Section~\ref{hitheta}.
\subsection{Linear representations of $FVB_n$}\label{linrepr}
The group $FVB_n$ of flat virtual braids is known to be linear \cite[Corollary 5.4]{BarBelDom}. In this section we construct linear representations of $FVB_n$ related with the representation $\theta:FVB_n\to {\rm Aut}(F_{2n})$. We use classical notation. Symbols ${\rm I}_n$ and $O_{n\times m}$ denote the identity $n\times n$ matrix and the $n\times m$ matrix with zero entries, respectively. If  $A$ is an $n\times n$ matrix and $B$ is an $m\times m$ matrix, then the symbol  $A\oplus B$ denotes the direct sum of the matrices $A$ and $B$, i.~e. the block-diagonal $(m+n)\times (m+n)$ matrix
$$
\newcommand{\tempa}{\multicolumn{1}{c|}{A}}
\newcommand{\tempb}{\multicolumn{1}{|c}{B}}
\begin{pmatrix}
\tempa&O_{n\times m}\\\cline{1-2}
O_{m\times n}&\tempb
\end{pmatrix}
.
$$
 
 There exists a classical procedure of how to construct specific mappings from braid-like groups to some matrix groups using rerepresentations of these braid-like groups by automorphisms of free groups using Fox's derivatives. Sometimes the resulting mappings from braid-like groups to matrix groups are representations. In particular,   using this procedure the Burau representation \cite[Section 3, Example 3]{Bir} and the Gassner representation \cite[Section 3, Example 4]{Bir} can be obtained. Let us recall this procedure. First, let us recall the  definition and the main properties of Fox's derivatives \cite[Chapter 3]{Bir} or \cite[Chapter 7]{CF}.
 
  Let $F_n$ be the free group of rank $n$ with the free generators $x_1, x_2, \ldots, x_n$. Let us consider the group ring $\mathbb{Z}F_n$ of
 the group $F_n$ over  the ring $\mathbb{Z}$ of integers. For $j=1, 2, \dots, n$ let us define the mapping
 $$
 \frac{\partial }{\partial x_j} : \mathbb{Z}F_n\longrightarrow \mathbb{Z}F_n.
 $$
By the definition the mapping $\frac{\partial }{\partial x_j}$ is defined on $F_n$ by the following three equalities
 \begin{alignat*}2
\frac{\partial x_i}{\partial x_j}&=
 \begin{cases}
 1, & \text{ if } i = j , \\
 0, & \text{ if } i\not= j,
 \end{cases}&&
 \\
\frac{\partial x_i^{-1}}{\partial x_j}&=
 \begin{cases}
 -x_i^{-1}, & \text{ if } i = j , \\
 0,         & \text{ if } i\not= j,
 \end{cases}&&\\
\frac{\partial (w v)}{\partial x_j} &= \frac{\partial w}{\partial x_j} (v)^{\tau } +
 w \frac{\partial v}{\partial x_j},&&w, v \in F_n,
 \end{alignat*}
 where $\tau : \mathbb{Z}F_n \longrightarrow
 \mathbb{Z}$ is the augmentation operation which
 sends all elements from $F_n$ to $1$. By the definition  the mapping $\frac{\partial }{\partial x_j}$ is extended from $F_n$ to $\mathbb{Z}F_n$ by the linearity in the following way
 $$
\frac{\partial }{\partial x_j} \left( \sum a_g g\right) =
 \sum a_g \frac{\partial g}{\partial x_j},~~~g\in F_n, a_g \in
 \mathbb{Z}.
 $$
 
 Let $\varphi:F_n\to G$ be a homomorphism from the group $F_n$ to some group $G$. For any element $x$ from $F_n$ denote by $x^{\varphi}$ the image of the element $x$ under the homomorphism $\varphi$. The homomorphism  $\varphi:F_n\to G$ can be extended to the homomorphism $\varphi:\mathbb{Z}F_n\to \mathbb{Z}G$ of group rings. Let $A_{\varphi}$ be a subgroup of ${\rm Aut}(F_n)$ which satisfies the condition $x^{\varphi} = (x^{\alpha})^{\varphi}$ for every $x\in F_n$, $\alpha \in A_{\varphi}$. If $\alpha \in
 A_{\varphi}$, then denote by $[\alpha]$  the $n \times n$
 matrix
\begin{equation}\label{hometheend}
 [\alpha] = \left[ \left( \frac{\partial (x_i^{\alpha })}{\partial
 	x_j} \right)^{\varphi} \right]_{i,j=1}^n
\end{equation}
 with entries in
 $\mathbb{Z}G$. The mapping $\alpha\mapsto [\alpha]$ defines the Magnus mapping
 $$
 \rho : A_{\varphi} \longrightarrow {\rm GL}_n
 (\mathbb{Z}G).
 $$
If $\psi:H\to{\rm Aut}(F_n)$ is the homomorphism of groups, then the superposition $\psi\rho$ is the mapping from $H$ to ${\rm GL}_n(\mathbb{Z}G)$
$$\psi\rho:H\to {\rm GL}_n(\mathbb{Z}G).$$
We say that this mapping $\psi\rho$ is obtained using Magnus approach \cite[Section~3.2]{Bir}. Summing up, in order to construct the mapping $\psi\rho:H\to {\rm GL}_n(\mathbb{Z}G)$ using Magnus approach one needs a homomorphism $\psi:H\to{\rm Aut}(F_n)$ and a homomorphism $\varphi:F_n\to G$.
 
Let us use the Magnus approach in order to construct the linear representations of the flat virtual braid group $FVB_n$. Denote by $G$ the free abelian group of rank $n$
$$G=\langle p_1,p_2,\dots,p_n,q_1,q_2,\dots,q_n~|~[p_i,p_j]=[q_i,q_j]=[p_i,q_j]=1~\text{ for}~i,j=1,2,\dots,n\rangle.$$
Let the map $\varphi:F_{2n}\to G$ is given by the equalities $\varphi(x_i)=p_i$, $\varphi(y_i)=q_i$ for $i=1,2,\dots,n$, where $x_1,x_2,\dots,x_n,y_1,y_2,\dots,y_n$ are the generators of $F_{2n}$. As a homomorphism $\psi$ given in Magnus approach let us take the representations $\theta:FVB_n\to{\rm Aut}(F_{2n})$ given in Section~\ref{hitheta}, so, now we have the representations $\theta:FVB_n\to{\rm Aut}(F_{2n})$ and the homomorphism $\varphi:F_{2n}\to G$, i.~e. we have everything in order to apply the Magnus approach and obtain the mapping 
$$
\Theta=\theta\rho : FVB_n \longrightarrow {\rm GL}_{2n}
(\mathbb{Z}G)={\rm GL}_{2n}
(\mathbb{Z}[p_1^{\pm1}, p_2^{\pm1},\dots,p_n^{\pm1},q_1^{\pm1}, q_2^{\pm1},\dots,q_n^{\pm1}]).
$$
Using the direct calculations in formula (\ref{hometheend}) we conclude that the mapping $\Theta$ acts on the generators $\sigma_1,\sigma_2,\dots,\sigma_{n-1},\rho_1,\rho_2,\dots,\rho_{n-1}$ in the following way
\begin{align*}
\Theta(\sigma_i):\begin{cases}
e_{i} \mapsto e_{i+1} + p_{i+1} f_{i+1},  \\
e_{i+1} \mapsto e_i - p_i q_{i+1}^{-1} f_{i+1}, 
\end{cases}&&\Theta(\rho_i):\begin{cases}
e_{i} \mapsto e_{i+1},   \\
e_{i+1} \mapsto e_i ,  \\
f_{i} \mapsto f_{i+1},   \\
f_{i+1} \mapsto f_i.
\end{cases}
\end{align*}
In order to find conditions under which the mapping $\Theta$ gives a representations, we have to check when the mapping $\Theta$ preserves the relations  (\ref{def2})-(\ref{def8}), (\ref{def1}) of the flat virtual braid group $FVB_n$. Using direct calculations it is easy to see that the mapping $\Theta$  preserves relations  (\ref{def2})-(\ref{def8}), (\ref{def1}) if and only if $p_1=p_2=\dots=p_n$ (we denote all of these elements by $p$), $q_1=q_2=\dots=q_n=1$. This fact implies the following theorem.

\begin{theorem}\label{lingoodrepr}Let $M$ be the free $\mathbb{Z}[p^{\pm1}]$-module with the free basis $e_1,e_2,\dots,e_n,f_1,f_2,\dots,f_n$. Then the map  $\Theta:\{\sigma_1,\sigma_2,\dots,\sigma_{n-1},\rho_1,\rho_2,\dots,\rho_{n-1}\}\to{\rm Aut}(M)$ given by the following formulas
\begin{align*}
\Theta(\sigma_i):\begin{cases}
e_{i} \mapsto e_{i+1} + p f_{i+1},  \\
e_{i+1} \mapsto e_i - p f_{i+1}, 
\end{cases}&&\Theta(\rho_i):\begin{cases}
e_{i} \mapsto e_{i+1},   \\
e_{i+1} \mapsto e_i ,  \\
f_{i} \mapsto f_{i+1},   \\
f_{i+1} \mapsto f_i.
\end{cases}
\end{align*}
induces the linear representation $\Theta:FVB_n\to{\rm Aut}(M)$. Moreover, this representation does not preserve the forbidden relations (\ref{forbidden}).
\end{theorem}

Often it is more convenient to work with matrices instead of working with actions on of the automorphisms on the basis elements. The following statement is the reformulation of Theorem~\ref{lingoodrepr} in terms of matrices. The matrices are written in the basis $e_1,f_1,e_2,f_2,\dots,e_n,f_n$ instead of $e_1,e_2,\dots,e_n,f_1,f_2,\dots,f_n$.
\begin{corollary}
	The map  $\Theta:\{\sigma_1,\sigma_2,\dots,\sigma_{n-1},\rho_1,\rho_2,\dots,\rho_{n-1}\}\to{\rm GL}_{2n}(\mathbb{Z}[p^{\pm1}])$ given by the following formulas
	\begin{align*}
	\Theta(\sigma_i)&={\rm I}_{2(i-1)}\oplus\begin{pmatrix}0&0&1&~~p\\
	0&1&0&~~0\\
	1&0&0&-p\\
	0&0&0&~~1
	\end{pmatrix}\oplus {\rm I}_{2(n-i-1)},\\
	\Theta(\rho_i)&={\rm I}_{2(i-1)}\oplus\begin{pmatrix}0&0&1&0\\
	0&0&0&1\\
	1&0&0&0\\
	0&1&0&0
	\end{pmatrix}\oplus {\rm I}_{2(n-i-1)}
	\end{align*}
	induces the linear representation $\Theta:FVB_n\to{\rm GL}_{2n}(\mathbb{Z}[p^{\pm1}])$. Moreover, this representation does not preserve the forbidden relations (\ref{forbidden}).
\end{corollary}
From the definition of $\Theta$ ($\Theta=\theta\rho$) it is clear that the kernel of $\theta$ is a subset in the kernel of $\Theta$, in particular,  the normal closure of the subgroup $\langle(\rho_i\sigma_{i+1})^6~|~i=1,2,\dots,n-2\rangle$ in $FVB_n$ belongs to the kernel of $\Theta$. Let us try to find a linear representation of the group $FVB_n$ such that the elements $(\rho_i\sigma_{i+1})^6$ donot belong to the kernel of this representation for $i=1,2,\dots,n-2$.  In the following theorem we construct a linear representation
 $$\Delta:FVB_n\to {\rm GL}_{3n}(\mathbb{Z})$$
such that the elements $(\rho_i\sigma_{i+1})^6$ donot belong to the kernel ${\rm ker}(\Delta)$.
\begin{theorem}\label{lingoodrepr2}Let $A$ be the free $\mathbb{Z}$-module with the basis $e_1,e_2,\dots,e_n,f_1,f_2,\dots,f_n,g_1,g_2,\dots,g_n$. Then the map  $\Delta:\{\sigma_1,\sigma_2,\dots,\sigma_{n-1},\rho_1,\rho_2,\dots,\rho_{n-1}\}\to{\rm Aut}(A)={\rm GL}_{3n}(\mathbb{Z})$ given by the following formulas
	\begin{align*}
	\Delta(\sigma_i):
	\begin{cases}
	e_i \mapsto e_{i+1}+g_{i},\\
	e_{i+1} \mapsto e_i-g_{i+1},\\
	g_i\mapsto g_{i+1},\\
	g_{i+1}\mapsto g_i,
	\end{cases}&&
	\Delta(\rho_i):
	\begin{cases}
	e_i \mapsto e_{i+1},\\
	e_{i+1} \mapsto e_i,\\
	f_i \mapsto f_{i+1},\\
	f_{i+1} \mapsto f_i,\\
	g_i\mapsto g_{i+1},\\
	g_{i+1}\mapsto g_i
	\end{cases}
	\end{align*}
	induces the linear representation $\Delta:FVB_n\to{\rm GL}_{3n}(\mathbb{Z})$. Moreover, this representation does not preserve the forbidden relations (\ref{forbidden}), and the elements $(\rho_i\sigma_{i+1})^6~(i=1,2,\dots,n-2)$ donot belong to the kernel ${\rm ker}(\Delta)$.
\end{theorem}
\begin{proof}In order to check that the map $$\Delta:\{\sigma_1,\sigma_2,\dots,\sigma_{n-1},\rho_1,\rho_2,\dots,\rho_{n-1}\}\to{\rm Aut}(A)={\rm GL}_{3n}(\mathbb{Z})$$
induces the representation $\Delta:FVB_n\to{\rm GL}_{3n}(\mathbb{Z})$ it is enough to verify that the map $\Delta$ preserves all defining relations  (\ref{def2})-(\ref{def8}), (\ref{def1}) of the flat virtual braid group $FVB_n$. This can be done by
direct calculations analougously to the proof of Theorem~\ref{goodrepr}, so, we will not do it here. The fact that $\Delta$ does not preserve the forbidden relations follows from the inequality 
 $$\Delta(\rho_i)\Delta(\sigma_{i+1})\Delta(\sigma_i)(f_i)=f_{i+1}\neq f_i=\Delta(\sigma_{i+1})\Delta(\sigma_i)\Delta(\rho_{i+1})(f_i).$$ 
 Finally, the fact that the elements $(\rho_i\sigma_{i+1})^6$ for $i=1,2,\dots,n-2$ donot belong to the kernel ${\rm ker}(\theta)$ follows from the equality  $(\Delta(\rho_i\sigma_{i+1}))^6(e_i)=e_i-2g_{i+1}+2g_{i+2}$.
\end{proof}
The following statement is the reformulation of Theorem~\ref{lingoodrepr2} in terms of matrices. The matrices are written in the basis $e_1,g_1,f_1,e_2,f_2,g_2,\dots,e_n,f_n,zg_n$ instead of $e_1,e_2,\dots,e_n$, $f_1,f_2,\dots,f_n$, $g_1,g_2,\dots,g_n$.
\begin{corollary}
	The map  $\Delta:\{\sigma_1,\sigma_2,\dots,\sigma_{n-1},\rho_1,\rho_2,\dots,\rho_{n-1}\}\to{\rm GL}_{3n}(\mathbb{Z})$ given by the following formulas
	\begin{align*}
	\Delta(\sigma_i)&={\rm I}_{3(i-1)}\oplus\begin{pmatrix}0&0&1&1&0&~~0\\
	0&1&0&0&0&~~0\\
	0&0&0&0&0&~~1\\
	1&0&0&0&0&-1\\
	0&0&0&0&1&~~0\\
	0&0&1&0&0&~~0
	\end{pmatrix}\oplus {\rm I}_{3(n-i-1)},\\
	\Delta(\rho_i)&={\rm I}_{3(i-1)}\oplus\begin{pmatrix}0&0&0&1&0&0\\
	0&0&0&0&1&0\\
	0&0&0&0&0&1\\
	1&0&0&0&0&0\\
	0&1&0&0&0&0\\
	0&0&1&0&0&0
	\end{pmatrix}\oplus {\rm I}_{3(n-i-1)}
	\end{align*}
induces the linear representation $\Delta:FVB_n\to{\rm GL}_{3n}(\mathbb{Z})$. Moreover, this representation does not preserve the forbidden relations (\ref{forbidden}), and the elements $(\rho_i\sigma_{i+1})^6~(i=1,2,\dots,n-2)$ donot belong to the kernel ${\rm ker}(\Delta)$.
\end{corollary}
\begin{question}{\rm Find the kernel ${\rm ker}(\Delta)$.} 
\end{question}
From Theorem~\ref{lingoodrepr2} it follows that there are elements which belong to the kernel of the representation $\Theta:FVB_n\to{\rm GL}_{2n}(\mathbb{Z}[p^{pm1}])$ and donot belong to the kernel of the  representation $\Delta:FVB_n\to{\rm GL}_{3n}(\mathbb{Z})$. In general the relations between the kernels ${\rm ker}(\Theta)$ and ${\rm ker}(\Delta)$ are not clear. In order to construct the linear representation which has all advantages of both representations  $\Theta$, $\Delta$ one can construct the representation $(\Theta\oplus\Delta):FVB_n\to {\rm GL}_{5n}(\mathbb{Z}[p^{\pm1}])$ which maps the braid $\alpha\in FVB_n$ to the matrix $\Theta(\alpha)\oplus\Delta(\alpha)$ from ${\rm GL}_{5n}(\mathbb{Z}[p^{\pm1}])$.

\subsection{Representation of $GVB_n$}\label{gausrepr}
Let us describe when the representation $\theta:FVB_n\to{\rm Aut}(F_{2n})$  introduced in Section~\ref{newwwrepp} induces  the representation of the group  $GVB_n$ of Gauss virtual braids. In order to do it let us understand under which conditions the map $\theta$ preserves the relations $\sigma_i\rho_i=\rho_i\sigma_i$ for $i=1,2,\dots,n-1$. Denote by $\theta(\sigma_i)=s_i$, $\theta(\rho_i)=r_i$ and let us calculate the automorphisms $r_is_i$, $s_ir_i$. From the direct calculations it follows that these automorphisms have the following forms
\begin{align*}
r_is_i:\begin{cases}
x_i\mapsto x_iy_{i+1}^{-1},\\
x_{i+1}\mapsto x_{i+1}y_{i+1},\\
y_i\mapsto y_{i+1},\\
y_{i+1}\mapsto y_i,
\end{cases}&&s_ir_i:\begin{cases}
x_i\mapsto x_iy_{i},\\
x_{i+1}\mapsto x_{i+1}y_{i}^{-1},\\
y_i\mapsto y_{i+1},\\
y_{i+1}\mapsto y_i.
\end{cases}
\end{align*}
From these formulas it follows that the representation $\theta:FVB_n\to{\rm Aut}(F_{2n})$ induces the representation of the Gauss virtual braid group if and only if $y_i=y_{i+1}^{-1}$ for $i=1,2,\dots,n-1$. If we denote by $y_1=z$, then for $i=2,3,\dots,n-1$ the equality $y_i=y_{i+1}^{-1}$ is equivalent to the equality $y_i=z^{(-1)^{i+1}}$. In these denotations the representation $\theta:FVB_n\to{\rm Aut}(F_{2n})$ induces the representation  $\theta_G:GVB_n\to{\rm Aut}(F_{n+1})$, where $F_{n+1}=\langle x_1,x_2,\dots,x_n,z\rangle$, which acts on the generators $\sigma_1,\sigma_2,\dots,\sigma_{n-1},\rho_1,\rho_2,\dots,\rho_{n-1}$ by the following rules
\begin{align}\label{gaussmainrepr}
\theta_G(\sigma_i):
\begin{cases}
x_i \mapsto x_{i+1}z^{(-1)^i},\\
x_{i+1} \mapsto x_iz^{(-1)^{i+1}},\\
\end{cases}&&
\theta_G(\rho_i):
\begin{cases}
x_i \mapsto x_{i+1},\\
x_{i+1} \mapsto x_i,\\
z\mapsto z^{-1}.
\end{cases}
\end{align}
Unfortunately, the representation $\theta_G:GVB_n\to{\rm Aut}(F_{n+1})$ preserves the forbidden relations~(\ref{forbidden}). Indeed, denote by $\theta_G(\sigma_i)=s_i$, $\theta_G(\rho_i)=r_i$ and let us calculate the automorphisms $r_is_{i+1}s_i$, $s_{i+1}s_ir_{i+1}$. From direct calculations it follows that 
\begin{align*}
r_is_{i+1}s_i&:\begin{cases}x_i\xmapsto{r_i}x_{i+1}\xmapsto{s_{i+1}}x_{i+2}z^{(-1)^{i+1}}\xmapsto{s_i}x_{i+2}z^{(-1)^{i+1}},\\
x_{i+1}\xmapsto{r_i}x_{i}\xmapsto{s_{i+1}}x_i\xmapsto{s_{i}}x_{i+1}z^{(-1)^i},\\
x_{i+2}\xmapsto{r_i}x_{i+2}\xmapsto{s_{i+1}} x_{i+1}z^{(-1)^{i+2}}\xmapsto{s_i}x_iz^{(-1)^{i+1}+(-1)^{i+2}}=x_i,\\
z\xmapsto{r_i}z^{-1}\xmapsto{s_{i+1}}z^{-1}\xmapsto{s_{i}}z^{-1},
\end{cases}\\
s_{i+1}s_ir_{i+1}&:\begin{cases}x_i\xmapsto{s_{i+1}}x_i\xmapsto{s_i}x_{i+1}z^{(-1)^{i}}\xmapsto{r_{i+1}}x_{i+2}z^{(-1)^{i+1}},\\
x_{i+1}\xmapsto{s_{i+1}}x_{i+2}z^{(-1)^{i+1}}\xmapsto{s_i}x_{i+2}z^{(-1)^{i+1}}\xmapsto{r_{i+1}}x_{i+1}z^{(-1)^{i+2}}=x_{i+1}z^{(-1)^{i}},\\
x_{i+2}\xmapsto{s_{i+1}}x_{i+1}z^{(-1)^{i+2}}\xmapsto{s_{i}}x_{i}z^{(-1)^{i+1}+(-1)^{i+2}}=x_i\xmapsto{r_{i+1}}x_i,\\
z\xmapsto{s_{i+1}}z\xmapsto{s_{i}}z\xmapsto{r_{i+1}}z^{-1}.
\end{cases}
\end{align*}
These equalities mean that the representation $\theta_G:GVB_n\to{\rm Aut}(F_{n+1})$ preserves the forbidden relations (\ref{forbidden}). Let us modify the representation $\theta_G:GVB_n\to{\rm Aut}(F_{n+1})$ to the representation which does not preserve the forbidden relations and still is very similar to the representation $\theta:FVB_n\to{\rm Aut}(F_{2n})$.

Let us consider the free group $F_{2n+1}$ with the free generators $x_1,x_2,\dots,x_n,y_1,y_2,\dots,y_n, z$. The following theorem is a kind of analogues to Theorem~\ref{goodrepr} for Gauss virtual braid groups.
\begin{theorem}\label{gaussmainrepr2th} Let the map $\theta_G^*:\{\sigma_1,\sigma_2,\dots,\sigma_{n-1},\rho_1,\rho_2,\dots,\rho_{n-1}\}\to{\rm Aut}(F_{2n+1})$ is given by the following formulas
\begin{align}\label{gaussmainrepr2}
\theta_G^*(\sigma_i):
\begin{cases}
x_i \mapsto x_{i+1}z^{(-1)^i},\\
x_{i+1} \mapsto x_iz^{(-1)^{i+1}},\\
\end{cases}&&
\theta_G^*(\rho_i):
\begin{cases}
x_i \mapsto x_{i+1},\\
x_{i+1} \mapsto x_i,\\
y_i\mapsto y_{i+1},\\
y_{i+1}\mapsto y_i,\\
z\mapsto z^{-1}.
\end{cases}
\end{align}
Then $\theta_G^*$ induces the representation $\theta_G^*:GVB_n\to{\rm Aut}(F_{2n+1})$. Moreover, this representation does not preserve the forbidden relations (\ref{forbidden}).
\end{theorem}
\begin{proof}
The proof completely repeats the proof of Theorem~\ref{goodrepr}.
\end{proof}
The problem of defining the kernel ${\rm ker}(\theta_G^*)$ of the representation $\theta_G^*:GVB_n\to{\rm Aut}(F_{2n+1})$ given in Theorem~\ref{gaussmainrepr2th} is easier then the problem of defining the kernel ${\rm ker}(\theta)$ of the representation $\theta:FVB_n\to{\rm Aut}(F_{2n})$ given in Theorem~\ref{goodrepr}, and it is solved in the following theorem.
\begin{theorem}\label{Gauusker}${\rm Im}(\theta_G^*)=S_n\times S_n$, ${\rm ker}(\theta_G^*)=GVP_n\cap GH_n$.
\end{theorem}
\begin{proof} First, let us prove that ${\rm Im}(\theta_G^*)=S_n\times S_n$. For $i=1,2,\dots,n-1$ denote by $\eta_i=\lambda_{i,i+1}=\sigma_i\rho_i$. It is clear that the group $GVB_n$ is generated by the elements $\sigma_1,\sigma_2,\dots,\sigma_{n-1},\eta_1,\eta_2,\dots,\eta_{n-1}$. From this fact it follows that ${\rm Im}(\theta_G^*)$ is generated by the elements $\theta_G^*(\sigma_1),\theta_G^*(\sigma_2),\dots,\theta_G^*(\sigma_{n-1}),\theta_G^*(\eta_1),\theta_G^*(\eta_2),\dots,\theta_G^*(\eta_{n-1})$. From the direct calculations it is easy to see that for $i=1,2,\dots,n-1$ the following formulas hold
	\begin{align}\label{easywin}
	\theta_G^*(\sigma_i):
	\begin{cases}
	x_i \mapsto x_{i+1}z^{(-1)^i},\\
	x_{i+1} \mapsto x_iz^{(-1)^{i+1}},\\
	\end{cases}&&
	\theta_G^*(\eta_i):
	\begin{cases}
	x_i \mapsto x_{i}z^{(-1)^{i+1}},\\
	x_{i+1} \mapsto x_{i+1}z^{(-1)^{i}},\\
	y_i\mapsto y_{i+1},\\
	y_{i+1}\mapsto y_i,\\
	z\mapsto z^{-1}.
	\end{cases}
	\end{align}
From formulas (\ref{easywin}) it is clear that the elements $\theta_G^*(\sigma_1),\theta_G^*(\sigma_2),\dots,\theta_G^*(\sigma_{n-1})$ generate the full permutation group $S_n$ in ${\rm Im}(\theta_G^*)$. Also from formulas (\ref{easywin}) it is clear that the elements $\theta_G^*(\eta_1),\theta_G^*(\eta_2),\dots,\theta_G^*(\eta_{n-1})$ generate the full permutation group $S_n$ in ${\rm Im}(\theta_G^*)$. Since the elements $\theta_G^*(\sigma_i)~(i=1,2,\dots,n-1)$ act non-identically on $x_1,x_2,\dots,x_n$, and identically on $y_1,y_2,\dots,y_n$, and the elements $\theta_G^*(\rho_i)~(i=1,2,\dots,n-1)$ act identically (modulo the normal closure of the element $z$) on $x_1,x_2,\dots,x_n$, and non-identically on $y_1,y_2,\dots,y_n$, the group $S_n$ generated by the elements $\theta_G^*(\sigma_1),\theta_G^*(\sigma_2),\dots,\theta_G^*(\sigma_{n-1})$, and the group $S_n$ generated by the elements $\theta_G^*(\eta_1),\theta_G^*(\eta_2),\dots,\theta_G^*(\eta_{n-1})$ have the trivial intersection. Finally using formulas (\ref{easywin}) and direct calculations it is easy to see that $\theta_G^*(\sigma_i)\theta_G^*(\rho_j)=\theta_G^*(\rho_j)\theta_G^*(\sigma_i)$ for arbitrary $i,j=1,2,\dots,n-1$. Hence, ${\rm Im}(\theta_G^*)$ is generated by two groups $S_n$ which commute with each other and which intersect trivially, therefore, ${\rm Im}(\theta_G^*)=S_n\times S_n$.

Let us prove that ${\rm ker}(\theta_G^*)=GVP_n\cap GH_n$. The inclusion ${\rm ker}(\theta_G^*)\leq GVP_n\cap GH_n$ can be proved in the same way to Proposition~\ref{kerisinintern}. Let us prove that the inclusion $GVP_n\cap GH_n\leq {\rm ker}(\theta_G^*)$ holds. 

Let $A$ be an element from the intersection $GVP_n\cap GH_n$. 
From the definition of the generators $x_{i,j}~(1\leq i<j\leq n-1)$ of $GH_n$ and the direct calculations it follows that the automorphisms $\theta_G^*(x_{i,j})~(1\leq i<j\leq n-1)$ fix the elements $y_1,y_2,\dots,y_n,z$. From this fact it follows that every element from $\theta_G^*(GH_n)$ fixes the elements $y_1,y_2,\dots,y_n,z$, in particular, since $A\in GH_n$, the automorphism $\theta_G^*(A)$ fixes the elements $y_1,y_2,\dots,y_n,z$.

From the direct calculations it follows that the image of the generator $\lambda_{i,j}~(1\leq i<j\leq n)$ of $GVP_n$ under the map $\theta_G^*$ has the following form
$$\theta^*_G(\lambda_{i,j}):
\begin{cases}
x_i \mapsto x_{j}z^{(-1)^{j+1}},\\
x_{j} \mapsto x_{i}z^{(-1)^{j}},\\
y_i\mapsto y_{j},\\
y_{j}\mapsto y_i,\\
z\mapsto z^{-1}.
\end{cases}$$
This automorphism can be expressed in the form $\theta^*_G(\lambda_{i,j})=\alpha_{i,j}\beta_{i,j}$, where the automorphisms $\alpha_{i,j},\beta_{i,j}$ have the following forms
\begin{align*}
\alpha_{i,j}:
\begin{cases}
x_i \mapsto x_{j}z^{(-1)^{j+1}},\\
x_{j} \mapsto x_{i}z^{(-1)^{j}},\\
z\mapsto z^{-1}.
\end{cases}&&\beta_{i,j}:
\begin{cases}
y_i\mapsto y_{j},\\
y_{j}\mapsto y_i.
\end{cases}
\end{align*}
Since these automorphisms act non-trivially on non-intersecting sets, it is obvious that for arbitrary $1\leq i<j\leq n$, $1\leq r<s\leq n$ the automorphisms $\alpha_{i,j}$, $\beta_{r,s}$ commute. It is clear that the automorphisms $\beta_{i,j}~(1\leq i<j\leq n)$ generate the full permutation group $S_n$, where the automorphism $\beta_{i,j}$ is the transposition $(i,j)$. Checking the defining relations of the permutation group $S_n$ given in Proposition~\ref{permtrans} it is easy to see that the map $\varphi$, which maps the automorphism $\beta_{i,j}$ to the automorphism $\alpha_{i,j}$ for $1\leq i<j\leq n$ gives a homomorphism from the group $S_n$ to the subgroup of ${\rm Aut}(F_{2n+1})$ generated by the automorphisms $\alpha_{i,j}$ for $1\leq i<j\leq n$. Since $A$ belongs to $GVP_n$, the elements $A$ can be written as a word over the elements $\lambda_{i,j}$
$$A=w(\lambda_{i,j},~1\leq i<j\leq n).$$
Since for arbitrary $1\leq i<j\leq n$, $1\leq r<s\leq n$ the automorphisms $\alpha_{i,j}$, $\beta_{r,s}$ commute, the element $\theta_G^*(A)$ can be written in the form
$$\theta_G^*(A)=w(\theta_G^*(\lambda_{i,j}),~1\leq i<j\leq n)=w(\alpha_{i,j},~1\leq i<j\leq n)w(\beta_{i,j},~1\leq i<j\leq n).$$
Since $A$ belongs to $GH_n$, the automorphism $\theta_G^*(A)$ fixes the elements $y_1,y_2,\dots,y_n,z$. Therefore $w(\beta_{i,j},~1\leq i<j\leq n)$ is the identity automorphism (since only it acts non-trivially only on $y_1,y_2,\dots,y_n$). On the other hand we have 
$$w(\alpha_{i,j},~1\leq i<j\leq n)=w(\varphi(\beta_{i,j}),~1\leq i<j\leq n)=\varphi(w(\beta_{i,j},~1\leq i<j\leq n)),$$
therefore $w(\alpha_{i,j},~1\leq i<j\leq n)$ is also the identity automorphism. Therefore the automorphism
$$\theta_G^*(A)=w(\alpha_{i,j},~1\leq i<j\leq n)w(\beta_{i,j},~1\leq i<j\leq n)$$
is identity, i.~e. $A$ belong to the kernel of $\theta_G^*$.
\end{proof}

\subsection{Other representations of $FVB_n$}\label{otherguys}
In this section we introduce several representations of the flat virtual braid group $FVB_n$ which we found together with the representation $\theta:FVB_n\to{\rm Aut}(F_{2n})$ introduced in Theorem~\ref{goodrepr}. Let $F_{2n}$ be the free group with the free generators $x_1,x_2,\dots,x_n,y_1,y_2,\dots,y_n$. For $m\in \mathbb{Z}$ denote by $\theta_{1,m}$ the following map from the set of generators of $FVB_n$ to ${\rm Aut}(F_{2n})$
\begin{align}\label{theta1}
\theta_{1,m}(\sigma_i):
\begin{cases}
x_i \mapsto x_{i+1}y_{i+1}^m,\\
x_{i+1} \mapsto x_iy_{i+1}^{-m},\\
\end{cases}&&
\theta_{1,m}(\rho_i):
\begin{cases}
x_i \mapsto x_{i+1},\\
x_{i+1} \mapsto x_i,\\
y_i \mapsto y_{i+1},\\
y_{i+1} \mapsto y_i.\\
\end{cases}
\end{align}
Let $F_{2n+1}$ be the free group with the free generators $x_1,x_2,\dots,x_n,y_1,y_2,\dots,y_n,z$. Denote by $\theta_2$ the following map from the set of generators of $FVB_n$ to ${\rm Aut}(F_{2n+1})$
\begin{align}\label{theta2}
\theta_2(\sigma_i):
\begin{cases}
x_i \mapsto x_{i+1}y_{i+1}^z,\\
x_{i+1} \mapsto x_i(y_{i+1}^{-1})^z,\\
\end{cases}&&
\theta_2(\rho_i):
\begin{cases}
x_i \mapsto x_{i+1},\\
x_{i+1} \mapsto x_i,\\
y_i \mapsto y_{i+1},\\
y_{i+1} \mapsto y_i.\\
\end{cases}
\end{align}
Let $F_n*(F_n\times \mathbb{Z})=\langle x_1,x_2,\dots,x_n,y_1,y_2,\dots,y_n,z~|~[y_i,z]=1, i=1,2,\dots,n\rangle$ be the free product of the free group $F_n$ generated by the elements $x_1,x_2,\dots,x_n$ and the group $F_n\times \mathbb{Z}$ generated by the elements $y_1,y_2,\dots,y_n,z$. Denote by $\theta_3, \theta_4$ the following maps from the set of generators of $FVB_n$ to ${\rm Aut}(F_n*(F_n\times \mathbb{Z}))$.
 \begin{align}
\label{theta3}&\theta_3(\sigma_i):
\begin{cases}
x_i \mapsto x_{i+1}y_{i+1},\\
x_{i+1} \mapsto x_iy_{i+1}^{-1},\\
\end{cases}&&
\theta_3(\rho_i):
\begin{cases}
x_i \mapsto x_{i+1}z,\\
x_{i+1} \mapsto x_iz^{-1},\\
y_i \mapsto y_{i+1},\\
y_{i+1} \mapsto y_i,\\
\end{cases}\\
\label{theta4}&\theta_4(\sigma_i):
\begin{cases}
x_i \mapsto x_{i+1}y_{i+1},\\
x_{i+1} \mapsto x_iy_{i+1}^{-1},\\
\end{cases}&&
\theta_4(\rho_i):
\begin{cases}
x_i \mapsto x_{i+1}^z,\\
x_{i+1} \mapsto x_i^{z^{-1}},\\
y_i \mapsto y_{i+1},\\
y_{i+1} \mapsto y_i.\\
\end{cases}
\end{align}
Let $G$, $H_1$, $H_2$ be groups, and $\varphi_1:G\to H_1$, $\varphi_2:G\to H_2$ be homomorphisms. The homomorphisms $\varphi_1,\varphi_2$ are said to be equivalent if ${\rm ker}(\varphi_1)={\rm ker}(\varphi_2)$. The main results of this section is the following theorem.
\begin{theorem}\label{cirtsim}The following statements hold.
	\begin{enumerate}\item The maps $\theta_{1,m}(m\in\mathbb{Z}),\theta_2,\theta_3,\theta_4$ given by formulas (\ref{theta1})-(\ref{theta4}) induce representations of the flat virtual braid group $FVB_n$. 
\item The representations $\theta_{1,m}(m\neq0),\theta_2,\theta_3,\theta_4$ are equivalent to the representation $\theta$ introduced in Theorem~\ref{goodrepr}.
\end{enumerate}
\end{theorem}
\begin{proof}In order to check that the maps $\theta_{1,m}(m\in\mathbb{Z}),\theta_2,\theta_3,\theta_4$ induce representations it is enough to verify that these maps preserve all defining relations  (\ref{def2})-(\ref{def8}), (\ref{def1}) of the flat virtual braid group $FVB_n$. This can be done by
	direct calculations analougously to the proof of Theorem~\ref{goodrepr}, so, we will not do it here.
	
Let us prove that the representations $\theta_{1,m}(m\neq0),\theta_2,\theta_3,\theta_4$ are equivalent to the representation $\theta$ introduced in Theorem~\ref{goodrepr}, i.~e. that the equalities
$${\rm ker}(\theta)={\rm ker}(\theta_{1,m})={\rm ker}(\theta_2)={\rm ker}(\theta_3)={\rm ker}(\theta_4)$$
hold. 

First, let us prove that for $m\neq 0$ the equality ${\rm ker}(\theta)={\rm ker}(\theta_{1,m})$ holds. In order to do it, in the free group $F_{2n}$ generated by the elements $x_1,x_2,\dots,x_n,y_1,y_2,\dots,y_n$ denote by $G$ the subgroup generated by the elements  $x_1,x_2,\dots,x_n,y_1^{m},y_2^m,\dots,y_n^m$. From the definition of the representation $\theta_{1,m}:FVB_n\to {\rm Aut}(F_{2n})$ given by formulas (\ref{theta1}) it is clear that $G$ is invariant under the action of $\theta_{1,m}(FVB_n)$. Denote by $\varphi:F_{2n}\to G$ the map which maps $x_i$ to $x_i$, and maps $y_i$ to $y_i^m$ for $i=1,2,\dots,n$. Since $m\neq 0$, the group $G$ is the free group with the free generators $x_1,x_2,\dots,x_n,y_1^{m},y_2^m,\dots,y_n^m$, and the map $\varphi$ induces the isomorphism from $F_{2n}$ to $G$. From the definitions of the representations $\theta_{1,m}, \theta:FVB_n\to{\rm Aut}(F_{2n})$, it is clear that for every braid $\alpha\in FVB_n$ the equality $\varphi\theta_{1,m}(\alpha)\varphi^{-1}=\theta(\alpha)$ holds. From this equality clearly follows that ${\rm ker}(\theta)={\rm ker}(\theta_{1,m})$.

Now let us prove the equality ${\rm ker}(\theta)={\rm ker}(\theta_2)$. The representation $\theta:FVB_n\to{\rm Aut}(F_{2n})$ given by formulas~(\ref{mainrepr}) can be extended to the representation $\theta_*:FVB_n\to{\rm Aut}(F_{2n+1})$ using the same formulas~(\ref{mainrepr}) fixing the generator $z$. It is clear that ${\rm ker}(\theta)={\rm ker}(\theta_*)$, so, let us prove that ${\rm ker}(\theta_*)={\rm ker}(\theta_2)$. In order to do it, denote by $\varphi$ the automorphism of $F_{2n+1}$ which maps $y_i$ to $y_i^z$ for $i=1,2,\dots,n$, and fixes the generators $x_1,x_2,\dots,x_n,z$. From the definitions of the representations $\theta_*, \theta_2:FVB_n\to{\rm Aut}(F_{2n+1})$, it is clear that for every braid $\alpha\in FVB_n$ the equality $\varphi^{-1}\theta_*(\alpha)\varphi=\theta_2(\alpha)$ holds. From this equality clearly follows that ${\rm ker}(\theta_*)={\rm ker}(\theta_2)$. Therefore ${\rm ker}(\theta)={\rm ker}(\theta_*)={\rm ker}(\theta_2)$.

The equalities ${\rm ker}(\theta)={\rm ker}(\theta_3)$, ${\rm ker}(\theta)={\rm ker}(\theta_4)$ have very similar proofs, so, here we will prove only the equality ${\rm ker}(\theta)={\rm ker}(\theta_3)$. The inclusion ${\rm ker}(\theta_3)\leq {\rm ker}(\theta)$ is clear, since if $\alpha\in {\rm ker}(\theta_3)$, then $\theta_3(\alpha)$ is identity for  every fixed $z$, in particular, for $z=1$. It means that $\alpha\in {\rm ker}(\theta)$. 

Let us prove the inclusion ${\rm ker}(\theta)\leq {\rm ker}(\theta_3)$. Let $\beta$ be the braid from ${\rm ker}(\theta)$. First, note that from the definitions of the representations $\theta,\theta_3$ it follows that for every braid $\alpha\in FVB_n$ the equalities $\theta(\alpha)(y_i)=\theta_3(\alpha)(y_i)$ holds for $i=1,2,\dots,n$. From this fact and the fact that $\beta$ belongs to ${\rm ker}(\theta)$ follows that $\theta_3(\beta)(y_i)=y_i$ for $i=1,2,\dots,n$.  The element $\theta_3(\rho_i)$ has the following form
\begin{align*}
\theta_3(\rho_i):
\begin{cases}
x_i \mapsto x_{i+1}z,\\
y_i \mapsto y_{i+1},\\
x_{i+1} \mapsto x_iz^{-1},\\
y_{i+1} \mapsto y_i.\\
\end{cases}
\end{align*}
From this formula it is clear that the indices of $x_j$ and $y_j$ (for $j=1,2,\dots,n$) increase and decrease simultaneously. From this fact and the fact that $\theta_3(\beta)(y_i)=y_i$ for $i=1,2,\dots,n$ follows that the total exponent of $z$ in the word $\theta_3(x_i)~(i=1,2,\dots,n)$ is equal to zero. Since all the elements $y_1,y_2,\dots,y_n$ commute with $z$, and the total exponent of $z$ in the word $\theta_3(x_i)~(i=1,2,\dots,n)$ is equal to zero, we have the equalities $\theta_3(x_i)=\theta(x_i)=x_i$ for $i=1,2,\dots,n$. Therefore $\beta$ belongs to the kernel ${\rm ker}(\theta)$.
\end{proof}

Recently in the paper \cite{Barsssin} the notion of the virtually symmetric representation of braid-like groups were defined. Theorem~\ref{cirtsim} and the definition of the representation $\theta$ given in Section~\ref{hitheta} say that all representations given by formulas (\ref{theta1})-(\ref{theta4}) are virtually symmetric.
\section{Normal generators of the groups $VP_n\cap H_n$, $FVP_n\cap FH_n$, $GVP_n\cap GH_n$}\label{intersectionss}
From Proposition~\ref{kerisinintern} and Theorem~\ref{Gauusker} we see that for studying kernels of representations of braid-like groups it is important to understand how the groups $VP_n\cap H_n$, $FVP_n\cap FH_n$, $GVP_n\cap GH_n$ look like. The group theoretic motivation for the studying of the groups $VP_n\cap H_n$, $FVP_n\cap FH_n$, $GVP_n\cap GH_n$ lies in the same field with the motivation of the papers \cite{Bar,BarBel,BarBelDom,Rab}. In this section we find the sets of normal generators of the groups $VP_n\cap H_n$, $FVP_n\cap FH_n$, $GVP_n\cap GH_n$. The main result of this section is the following theorem.
\begin{theorem}\label{intersect}Let $G$ be the subgroup of $VB_n$ generated by the elements 
	\begin{align}
\notag	&\lambda_{s,t}\lambda_{t,s}^{-1},&&1\leq s<t\leq n,\\
\notag&\lambda_{t,s}^2,&&1\leq s<t\leq n,\\
\label{gvph}&\lambda_{t,s}\lambda_{r,q}\lambda_{t,s}^{-1}\lambda_{r,q}^{-1},&&(t-r)(t-q)(s-r)(s-q)>0,\\ 
\notag&\lambda_{t,s}\lambda_{s,r}\lambda_{t,s}^{-1}\lambda_{t,r}^{-1},&&n\geq t>s>r\geq 1,\\
\notag&\lambda_{t,s}\lambda_{s,r}\lambda_{t,r}^{-1}\lambda_{s,r}^{-1},&&n\geq t>s>r\geq 1.
\end{align}
Then the normal closure of $G$ in $VP_n$ coincides with $VP_n\cap H_n$.
\end{theorem}
\begin{proof}Denote by $H$ the normal closure of $G$ in $VP_n$, and let us prove that $H=VP_n\cap H_n$. First, let us prove that $H\leq VP_n\cap H_n$. Since $VP_n\cap H_n$ is a normal subgroup of $VP_n$, in order to prove that $H$ is a subset of $VP_n\cap H_n$ it is enough to prove that $G$ is a subset of $VP_n\cap H_n$. The inclusion $G\leq VP_n$ follows from the fact that the generators of $G$ are written in terms of generators $\lambda_{i,j}$ of $VP_n$. In order to prove that $G\leq H_n$ let us consider the homomorphism $\iota_2:VB_n\to S_n$ which maps the generator $\sigma_i$ to the unit element of $S_n$, and maps the generator $\rho_i$ to the transposition $(i, i+1)$ for $i=1,2,\dots,n-1$. From the definition of the elements $\lambda_{i,j}$ it is clear that the homomorphism $\iota_2$ maps the element $\lambda_{i,j}$ to the transposition $(i,j)$. From this fact and formulas (\ref{gvph}) which define the generators of the group $G$ it is clear that all generators of $G$ belong to the kernel of $\iota_2$. Hence, the group $G$ is a subset of ${\rm ker}(\iota_2)=H_n$, and therefore $H\leq VP_n\cap H_n$.
	
Let us prove that $VP_n\cap H_n\leq H$. In order to do it, consider the quotient 
$$(VP_n\cap H_n)/H=(VP_n/H)\cap (H_n/H).$$ 
By the definition of the quotient group, the group $VP_n/H$ can be obtained from the group $VP_n$ adding new relations (\ref{gvph}). Therefore, from Proposition~\ref{vpngen}(1) it follows that the group $VP_n/H$ admits a presentation with the generators $\lambda_{k,l}~(1 \le k \not= l \le n)$, relations
\begin{align}
\label{one1}\lambda_{i,j}\lambda_{k,l} &= \lambda_{k,l}\lambda_{i,j}, \\
\label{one2}\lambda_{k,i}\lambda_{k,j}\lambda_{i,j} &= \lambda_{i,j}\lambda_{k,j}\lambda_{k,i},
\end{align}
where distinct letters denote distinct indices, and additional relations
\begin{align}
\label{one3}\lambda_{s,t}&=\lambda_{t,s},&&1\leq s<t\leq n,\\
\label{one4}\lambda_{t,s}^2&=1,&&1\leq s<t\leq n,\\
\label{one5}\lambda_{t,s}\lambda_{r,q}&=\lambda_{r,q}\lambda_{t,s},&&(t-r)(t-q)(s-r)(s-q)>0,\\ 
\label{one6}\lambda_{t,s}\lambda_{s,r}&=\lambda_{t,r}\lambda_{t,s},&&n\geq t>s>r\geq 1,\\
\label{one7}\lambda_{t,s}\lambda_{s,r}&=\lambda_{s,r}\lambda_{t,r},&&n\geq t>s>r\geq 1.
\end{align}
From Proposition~\ref{permtrans} it follows that the group with the generators $\lambda_{k,l}~(1 \le k \not= l \le n)$ and defining relations (\ref{one3})-(\ref{one7}) is the permutation group $S_n$, where $\lambda_{k,l}$ denotes the transposition $(k,l)$. If for $1 \le k \not= l \le n$ we write the transposition $(k,l)$ instead of $\lambda_{k,l}$ in relations (\ref{one1}), (\ref{one2}), then we get correct equalities. This fact means that relations (\ref{one1}), (\ref{one2}) follow from relations (\ref{one3})-(\ref{one7}), hence, the quotient $VP_n/H$ admits a presentation with the generators $\lambda_{k,l}~(1 \le k \not= l \le n)$ and relations
\begin{align*}
\lambda_{s,t}&=\lambda_{t,s},&&1\leq s<t\leq n,\\
\lambda_{t,s}^2&=1,&&1\leq s<t\leq n,\\
\lambda_{t,s}\lambda_{r,q}&=\lambda_{r,q}\lambda_{t,s},&&(t-r)(t-q)(s-r)(s-q)>0,\\ 
\lambda_{t,s}\lambda_{s,r}&=\lambda_{t,r}\lambda_{t,s},&&n\geq t>s>r\geq 1,\\
\lambda_{t,s}\lambda_{s,r}&=\lambda_{s,r}\lambda_{t,r},&&n\geq t>s>r\geq 1.
\end{align*}
From Proposition~\ref{permtrans} it follows that the group $VP_n/H$ is isomorphic to the full permutation group $S_n$, where the isomorphism maps the element $\lambda_{k,l}~(1 \le k \not= l \le n)$ to the transposition $(k,l)$.

Let $x$ be an element from $VP_n\cap H_n$. Denote by $\overline{x}$ the image of $x$ in $(VP_n\cap H_n)/H$. Since $(VP_n\cap H_n)/H=(VP_n/H)\cap (H_n/H)$, we have $\overline{x}\in VP_n/H$. Since $VP_n/H$ is isomorphic to the full permutation group $S_n$, where the isomorphism maps the element $\lambda_{k,l}~(1 \le k \not= l \le n)$ to the transposition $(k,l)$, and  $\overline{x}\in VP_n/H$, the element $x$ can be written in the form $x=yz$, where $y$ is the element from the preimage of $S_n$ in $VP_n$ and $z\in H$. Since $x\in H_n$ the image of $x$ under the the homomorphism $\iota_2:VB_n\to S_n$ which maps the generator $\sigma_i$ to the unit element of $S_n$, and maps the generator $\rho_i$ to the transposition $(i, i+1)$ for $i=1,2,\dots,n-1$ has to be identity. Therefore we have the equality 
$$1=\iota_2(x)=\iota_2(yz)=\iota_2(y)\iota_2(z).$$
 Since $z\in H\leq H_n$, we have $\iota_2(z)=1$, therefore $\iota_2(y)=1$. Since $\iota_2$ maps  the element $\lambda_{k,l}$ to the transposition $(k,l)$ and $y$ belongs to the preimage of $S_n$ in $VP_n$, the equality $\iota_2(y)=1$ implies that $y=1$. Hence, $x=z$, therefore $VP_n\cap H_n\leq H$ and $H=VP_n\cap H_n$.
\end{proof}
The group $FVB_n$ is the quotient of  $VB_n$ by the relations $\sigma_i^2=1$ for $i=1,2,\dots,n-1$. In this group the equality $\lambda_{t,s}=\lambda_{s,t}^{-1}$ holds for all $t,s$. The group $GVB_n$ is the quotient of the group $FVB_n$ by the relations $\sigma_i\rho_i=\rho_i\sigma_i$ for $i=1,2,\dots,n-1$. in this group the equality $\lambda_{t,s}^2=1$ holds for all $t,s$. These facts together with Theorem~\ref{intersect} imply the following corollary.
\begin{corollary}The following statements hold.
	\begin{enumerate}
		\item Let $G$ be the subgroup of $FVB_n$ generated by the elements 
		\begin{align}
		\notag	&\lambda_{s,t}\lambda_{t,s}^{-1},&&1\leq s<t\leq n,\\
		\notag&\lambda_{t,s}\lambda_{r,q}\lambda_{t,s}^{-1}\lambda_{r,q}^{-1},&&(t-r)(t-q)(s-r)(s-q)>0,\\ 
		\notag&\lambda_{t,s}\lambda_{s,r}\lambda_{t,s}^{-1}\lambda_{t,r}^{-1},&&n\geq t>s>r\geq 1,\\
		\notag&\lambda_{t,s}\lambda_{s,r}\lambda_{t,r}^{-1}\lambda_{s,r}^{-1},&&n\geq t>s>r\geq 1.
		\end{align}
		Then the normal closure of $G$ in $FVP_n$ coincides with $FVP_n\cap FH_n$.
		\item Let $G$ be the subgroup of $GVB_n$ generated by the elements 
		\begin{align}
		\notag&\lambda_{t,s}\lambda_{r,q}\lambda_{t,s}^{-1}\lambda_{r,q}^{-1},&&(t-r)(t-q)(s-r)(s-q)>0,\\ 
		\notag&\lambda_{t,s}\lambda_{s,r}\lambda_{t,s}^{-1}\lambda_{t,r}^{-1},&&n\geq t>s>r\geq 1,\\
		\notag&\lambda_{t,s}\lambda_{s,r}\lambda_{t,r}^{-1}\lambda_{s,r}^{-1},&&n\geq t>s>r\geq 1.
		\end{align}
		Then the normal closure of $G$ in $GVP_n$ coincides with $GVP_n\cap GH_n$.
	\end{enumerate}
\begin{question}{\rm Find the set of generators and the set of defining relations for groups $VP_n\cap H_n$, $FVP_n\cap FH_n$, $GVP_n\cap GH_n$.}
\end{question}	
\end{corollary}
\section{Group invariant for flat welded links}\label{invvvv}
In the papers \cite{BarNas1, BarNas2} the notion of a multi-switch was introduced. Using this notion one can construct representations for virtual braids  \cite{BarNas1} and invariants for virtual links  \cite{BarNas2}. In this section we recall the notion of a multi-switch, and using this notion and the representation $\theta:FVB_n\to {\rm Aut}(F_{2n})$ constructed in Section~\ref{hitheta} we introduce a group invariant for flat welded links. 

Let $X$ be a non empty set, and $X_1, X_2, \dots, X_m$ be non-empty subsets of $X$. One can identify the sets
$(X\times X_1\times X_2\times\dots\times X_m)^2$ and $X^2\times X_1^2\times X_2^2\times\dots\times X_m^2$ denoting the ordered pair $(A,B)$ of the elements $A=(a_0,a_1,\dots,a_m), B=(b_0,b_1,\dots,b_m)$ from $X\times X_1\times X_2\times\dots\times X_m$ by $(A,B)=(a_0,b_0; a_1,b_1;a_2,b_2;\dots;a_m,b_m)$.

A map $S:X^2\times X_1^2\times \dots\times X_m^2\to X^2\times X_1^2\times \dots\times X_m^2$ is called an $(m+1)$-switch (or a multi-switch), if $S$ satisfies the equality $(S\times id)(id\times S)(S\times id)=(id\times S)(S\times id)(id\times S)$, where $id$ denotes the identity on $X\times X_1\times X_2\times\dots\times X_m$, and $S$ has the form
$$S(c_0,c_1,\dots,c_m)=(S_0(c_0,c_1,\dots,c_m),S_1(c_1),\dots,S_m(c_m))$$
for $c_0\in X^2$, $c_i\in X_i^2$ for $i=1,2,\dots,m$,
and the maps $S_0, S_1,\dots,S_m$ 
\begin{align}
\notag S_0&:X^2 \times X_1^2 \times \dots \times X_m^2 \to X^2,\\
\notag S_i&:X_i^2  \to X_i^2,~\text{for}~i = 1, 2, \dots, m.
\end{align}
The fact that $S$ is an $(m+1)$-switch on $X$ defined by the maps $S_0,S_1,\dots,S_m$ is denoted by $S=(S_0,S_1,\dots,S_m)$. 

Let $S, V: X^2\times X_1^2\times X_2^2\times\dots\times X_m^2\to X^2\times X_1^2\times X_2^2\times\dots\times X_m^2$ be two $(m+1)$-switches on $X$. The ordered pair $(S,V)$ is called a virtual $(m+1)$-switch on $X$, if $V^2$ is the identity on $X^2\times X_1^2\times X_2^2\times\dots\times X_m^2$, and the equality
$
(id \times V) (S \times id)(id \times V)  =   (V \times id)(id \times S)(V \times id)
$ holds. Let us show how virtual multi-switches can be used for constructing representations of virtual braids.

Let $X$ be an algebraic system generated by the elements $\{x^{i}_j~|~i=0,1,\dots,m,j=1,2,\dots,n\}$, such that $\{x^{i}_{1}, x^{i}_2,\dots,x^{i}_{n}\}\cap\{x^{j}_{1}, x^{j}_2,\dots,x^{j}_{n}\}=\varnothing$ for $i\neq j$. For $i=0,1,2,\dots,m$ let $X_i$ be a subset of $X$ which contains the elements $x^{i}_{1},x^{i}_{2},\dots,x^{i}_{n}$. Let $S=(S_0,S_1,\dots,S_m)$, $V=(V_0,V_1,\dots,V_m)$ be a virtual $(m+1)$-switch on $X$ such that
\begin{align}
\notag S_0=(S_0^l,S_0^r), V_0=(V_0^l,V_0^r)&:X^2 \times X_1^2 \times X_2^2 \times \dots \times X_m^2 \to X^2,\\
\notag S_i=(S_i^l,S_i^r), V_i=(V_i^l,V_i^r)&:X_i^2  \to X_i^2,~\text{for}~i = 1, 2, \dots, m,
\end{align}
and for $i=0,1,\dots,m$ the images of maps $S_i^l,S_i^r, V_i^l,V_i^r$  are words over its arguments in terms of operations of $X$. For $j=1,2,\dots,n-1$ denote by $R_j,G_j$ the following maps from the set of generators  $\{x^{i}_j~|~i=0,1,\dots,m,j=1,2,\dots,n\}$ of $X$ to $X$
\begin{align}
\label{fj} R_j :&\begin{cases}
x^0_{j} \mapsto S_0^l(x^0_{j}, x^0_{j+1}, x^1_{j}, x^1_{j+1}, \dots, x^m_{j}, x^m_{j+1}),\\
x^0_{j+1} \mapsto S_0^r(x^0_{j}, x^0_{j+1}, x^1_{j}, x^1_{j+1}, \dots, x^m_{j}, x^m_{j+1}),\\
x^1_{j} \mapsto S_1^l(x^1_{j}, x^1_{j+1}),\\
x^1_{j+1} \mapsto S_1^r(x^1_{j}, x^1_{j+1}), \\
~~\vdots\\
x^m_{j} \mapsto S_m^l(x^m_{j}, x^m_{j+1}),\\
x^m_{j+1} \mapsto S_m^r(x^m_{j}, x^m_{j+1}),
\end{cases}\end{align}
\begin{align}
\label{gj} G_j :&\begin{cases}
x^0_{j} \mapsto V_0^l(x^0_{j}, x^0_{j+1}, x^1_{j}, x^1_{j+1}, \dots, x^m_{j}, x^m_{j+1}),\\
x^0_{j+1} \mapsto V_0^r(x^0_{j}, x^0_{j+1}, x^1_{j}, x^1_{j+1}, \dots, x^m_{j}, x^m_{j+1}),\\
x^1_{j} \mapsto V_1^l(x^1_{j}, x^1_{j+1}),\\
x^1_{j+1} \mapsto V_1^r(x^1_{j}, x^1_{j+1}), \\
~~\vdots\\
x^m_{j} \mapsto V_m^l(x^m_{j}, x^m_{j+1}),\\
x^m_{j+1} \mapsto V_m^r(x^m_{j}, x^m_{j+1}),
\end{cases}
\end{align}
where all generators which are not explicitly written in $R_j, G_j$ are fixed. If for $j=1,2,\dots,n-1$ the maps $R_j, G_j$ induce automorphisms of $X$, then we say that $(S,V)$ is an automorphic virtual multi-switch (shortly, AVMS) with respect to the set of generators $\{x^i_{j}~|~i=0,1,\dots,m,j=1,2,\dots,n\}$. The following result is proved in \cite[Theorem~2]{BarNas1}.

\begin{proposition}\label{vautrepr}
	Let $(S, V)$ be an AVMS on $X$
	with respect to the set of generators $\{x^i_{j}~|~i=0,1,\dots,m,j=1,2,\dots,n\}$. Then the map
	$$
	\varphi_{S,V} : VB_n \to {\rm Aut}(X)
	$$
	which is defined on the generators of $VB_n$ as
	\begin{align}
	\notag\varphi_{S,V}(\sigma_j) = R_j,&&\varphi_{S,V}(\rho_j) = G_j,&&{\text for}~j = 1, 2, \dots, n-1,
	\end{align}
	where $R_j$, $G_j$ are defined by equalities (\ref{fj}), (\ref{gj}),
	is a representation of $VB_n$.
\end{proposition}

Let $X$ be the free group $F_{2n}$ with the free generators $x_1^0,x_2^0,\dots,x_n^0,x_1^1,x_2^1,\dots,x_n^1$. Denote by $X_0$ the subgroup of $X$ generated by the elements  $x_1^0,x_2^0,\dots,x_n^0$, and by $X_1=\{x_1^1,x_2^1,\dots,x_n^1\}$. Let $S,V:X^2\times X_1^2\to X^2\times X_1^2$ be the maps given by the following formulas
\begin{align}\label{thetaswitch}
S(a,b;x,y)=(by,ay^{-1};x,y),&&V(a,b;x,y)=(b,a;y,x)
\end{align}
for $a,b\in X$, $x,y\in X_1$. It is easy to check that $(S,V)$ is a virtual $2$-switch on $X$. If we apply Proposition~\ref{vautrepr} to this multi-switch, then we get the representation $\varphi_{S,V}=\theta:FVB_n\to{\rm Aut}(F_{2n})$ introduced in Section~\ref{hitheta} (where $x_i$ is changed by $x_i^0$, and $y_i$ is changed by $x_i^1$).

Let us show how this representation can be used for constructing an invariant for flat welded links. In order to do it, let us recall the definition of a biquandle. A biquandle $BQ$ is an algebraic system with two binary algebraic operations $(a, b) \mapsto a^b$, $(a, b) \mapsto a_b$ which satisfy the following axioms.
\begin{enumerate}
	\item For all $a\in BQ$ the maps $f^a,f_a:BQ\to BQ$ given by the equalities $f^a(x)=x^a$, $f_a(x)=x_a$ for $x\in BQ$ are bijections on $BQ$.
	\item For all $a\in BQ$ the equalities $a^{a^{-1}}=a_{a^{a^{-1}}}$,  $a_{a^{-1}}=a^{a_{a^{-1}}}$ hold.
	\item For all $a,b,c\in Q$ the equalities  $a^{bc}=a^{c_bb^c}$, $a_{bc}=a_{c^bb_c}$, $(a_b)^{c_{b^a}}=(a^{c})_{b^{c_a}}$ hold.
\end{enumerate}

Let $S:(X\times X_1\times \dots\times X_m)^2\to (X\times X_1\times \dots\times X_m)^2$ be a multi-switch on $X$, such that $S(A,B)=(S^l(A,B), S^r(A,B))$ for $A,B\in X\times X_1\times \dots\times X_m$. Denote by
\begin{align}
\notag A_B=S^l(B,A),&& A^B=S^r(A,B).
\end{align}
If the algebraic system $X\times X_1\times \dots\times X_m$ with the operations $(A,B)\mapsto A^B$, $(A,B)\mapsto A_B$ is a biquandle, then the multi-switch $S$ is called a biquandle multi-switch. 

Let $X$ be an algebraic system generated by the elements $\{x^{i}_j~|~i=0,1,\dots,m,j=1,2,\dots,n\}$, such that $\{x^{i}_{1}, x^{i}_2,\dots,x^{i}_{n}\}\cap\{x^{j}_{1}, x^{j}_2,\dots,x^{j}_{n}\}=\varnothing$ for $i\neq j$. For $i=0,1,2,\dots,m$ let $X_i$ be a subset of $X$ which contains elements $x^{i}_{1},x^{i}_{2},\dots,x^{i}_{n}$. Let $S=(S_0,S_1,\dots,S_m)$, $V=(V_0,V_1,\dots,V_m)$ be an automorphic virtual $(m+1)$-switch on $X$ with respect to the set of generators $\{x^{i}_j~|~i=0,1,\dots,m,j=1,2,\dots,n\}$. From Proposition~\ref{vautrepr} it follows that one can construct a representation $\varphi_{S,V}:VB_n\to{\rm Aut}(X)$. For a braid $\beta\in VB_n$ denote by $X_{S,V}(\beta)$ the quotient of the algebraic system $X$ by the relations $\varphi_{S,V}(x_i^j)=x_i^j$ for  $i=0,1,\dots,m,j=1,2,\dots,n$. The following statement follows from  \cite[Theorems~1,~3]{BarNas2}.

\begin{proposition}\label{invariant}If $S,V$ are biquandle multi-switches on $X$, then the algebraic system $X_{S,V}(\beta)$ depends not on the braid $\beta$, but on the closure $\widehat{\beta}$ of the braid $\beta$.
\end{proposition}
 
From the Alexander theorem for virtual links \cite[Proposition~3]{Kam} it follows that every virtual link diagram $D$ is equivalent to the closure of some virtual braid $\beta$ from $VB_n$. From Proposition~\ref{invariant} it follows that $X_{S,V}(\beta)$ can be written as $X_{S,V}(D)$ since $X_{S,V}(\beta)$ depend not on $\beta$, but on its closure $D$. Therefore, if $S,V$ are biquandle multi-switches on $X$, then the map 
$$D\mapsto\beta\mapsto X_{S,V}(\beta)=X_{S,V}(D)$$
from the set of virtual link diagrams to the set of algebraic systems from the same category as $X$, is an invariant for virtual links.

Since the representation $\theta:FVB_n\to {\rm Aut}(F_{2n})$ can be obtained using Proposition~\ref{vautrepr} and the virtual 2-switch given by formulas (\ref{thetaswitch}), one can try to apply Proposition~\ref{invariant} in order to construct an invariant for flat virtual links. Unfortunately, Proposition~~\ref{invariant} cannot be directly applied since the 2-switch $S$ given by formulas (\ref{thetaswitch}) is not a biquandle $2$-switch (the first axiom of a biquandle does not hold). However, if we let $x_1^1=x_2^1=\dots=x_n^1$, then $(S,V)$ becomes a biquandle $2$-switch. From this fact and Proposition~\ref{invariant} we have the following statement.
\begin{theorem}\label{lastinv}For a braid $\beta\in FVB_n$ denote by $G(\beta)$ the quotient of the free group $F_{2n}$ with the free generators $x_1,x_2,\dots,x_n,y_1,y_2,\dots,y_n$ by the relations $\theta(\beta)(x_i)=x_i$, $\theta(\beta)(y_i)=y_i$ for $i=1,2,\dots,n$, and the relations $y_1=y_2=\dots=y_n$.  Then $G(\beta)$ depends not on the braid $\beta$, but on the closure $\widehat{\beta}$ of the braid $\beta$.
\end{theorem}
The group $G(\beta)$ from Theorem~\ref{lastinv} is the algebraic system $X_{S,V}(\beta)$ from Proposition~\ref{invariant} where the virtual $2$-switch $(S,V)$ is given by formulas (\ref{thetaswitch}). Since $\theta:FVB_n\to{\rm Aut}(F_{2n})$ is the representation of the flat virtual braid group, the group $G(\beta)$ is an invariant for flat virtual links. If in formulas (\ref{mainrepr}) which define the representation $\theta:FVB_n\to{\rm Aut}(F_{2n})$ we put $y_1=y_2=\dots=y_n$, then the resulting representation preserves the forbidden relations. Therefore, the group $G(\beta)$ gives an invariant for flat welded links.

As an example, let us calculate the group $G(\beta)$ in the case when  $\beta=(\rho_1\sigma_1)^m$ for some $m\geq0$. If $m=1$, then the closure of the braid $\beta$ is the flat virtual Hopf link. The braid $\beta$ belongs to $FVB_2$, therefore, by the definition, the group $G(\beta)$ is the quotient of the group $F_4=\langle x_1,x_2,y_1,y_2\rangle$ by the relations $\theta(\beta)(x_i)=x_i$, $\theta(\beta)(y_i)=y_i$ for $i=1,2$, and the relation $y_1=y_2$.
$$
G(\beta)=\langle x_1,x_2,y_1,y_2~|~\theta(\beta)(x_1)=x_1, \theta(\beta)(x_2)=x_2, \theta(\beta)(y_1)=y_1, \theta(\beta)(y_1)=y_1, y_1=y_2\rangle
$$
Using the relation $y_1=y_2$, denoting by $y=y_1=y_2$ let us rewrite the group $G(\beta)$ in the following form
$$
G(\beta)=\langle x_1,x_2,y~|~\theta(\beta)(x_1)=x_1, \theta(\beta)(x_2)=x_2, \theta(\beta)(y)=y\rangle.
$$
From the direct calculations it follows that the automorphism $\theta(\beta)=\theta((\sigma_1\rho_1)^m)$ provided by the equality $y_1=y_2=y$ has the following form 
$$\theta(\sigma_1\rho_1)^m:\begin{cases}
x_1\mapsto x_1y^m,\\
x_2\mapsto x_2y^{-m},\\
y\mapsto y.
\end{cases}$$
Therefore, the group $G(\beta)$ can be rewritten in the following form
$$G(\beta)=\langle x_1,x_2, y~|~x_1y^m=x_1, x_2y^{-m}=x_2\rangle=\langle x_1,x_2,y~|~y^m=1\rangle=F_2*\mathbb{Z}^m.$$
From this equality it follows that the group invariant introduced in Theorem~\ref{lastinv} destinguishes the closures of the braids $(\rho_1\sigma_1)^m$ for different $m\geq 0$. In particular, it distinguishes the flat virtual Hopf link from the flat virtual trivial link with two components.

~\\

~\\
Valeriy Bardakov$^{*,\dag, \ddag}$ (bardakov@math.nsc.ru), 
Bogdan Chuzhinov$^*$ (b.chuzhinov@g.nsu.ru),\\
Ivan  Emel'yanenkov$^*$ (i.emelianenkov@g.nsu.ru),
Maxim Ivanov$^*$ (m.ivanov2@g.nsu.ru),\\
Elizaveta Markhinina$^*$ (e.markhinina@g.nsu.ru),Timur Nasybullov$^{*,\dag, \ddag}$ (timur.nasybullov@mail.ru),
Sergey Panov$^*$ (s.panov@g.nsu.ru), Nina Singh$^{\ae}$ (ninok.singh@gmail.com),\\
Sergey Vasyutkin$^*$  (s.vasyutkin@g.nsu.ru),
Valeriy Yakhin$^*$ (v.yakhin@g.nsu.ru),\\
 Andrei Vesnin$^{*,\dag, \ddag}$ (vesnin@math.nsc.ru)

~\\
$^*$ Novosibirsk State University, Pirogova 1, 630090 Novosibirsk, Russia,\\
$^{\dag}$ Tomsk State University, pr. Lenina 36, 634050 Tomsk, Russia,\\
$^{\ddag}$ Sobolev Institute of Mathematics, Acad. Koptyug avenue 4, 630090 Novosibirsk, Russia,\\
$^{\ae}$ Saint Petersburg State Electrotechnical University, ul. Professora Popova 5, 197376 Saint Petersburg, Russia.

\end{document}